\documentclass[12pt]{amsart}
\usepackage{amsmath, amssymb, latexsym, amsthm, amscd, stmaryrd, amsfonts, mathrsfs}
\usepackage[title, titletoc, toc]{appendix}
\usepackage[all]{xy}
\usepackage{soul}

\theoremstyle{definition}
\newtheorem{Def}[subsubsection]{Definition}
\newtheorem{rem}[subsubsection]{Remark}
\theoremstyle{plain}
\newtheorem{prop}[subsubsection]{Proposition}
\newtheorem{thm}[subsubsection]{Theorem}
\newtheorem{lem}[subsubsection]{Lemma}
\newtheorem{cor}[subsubsection]{Corollary}

\newcommand{\mbf}{\mathbf}
\newcommand{\mfk}{\mathfrak}
\newcommand{\mscr}{\mathscr}
\newcommand{\mcal}{\mathcal}
\newcommand{\mbb}{\mathbb}
\newcommand{\mrm}{\mathrm}
\newcommand{\A}{\mathfrak a}
\newcommand{\Aa}{\mathfrak a}
\newcommand{\B}{\mathfrak b}
\newcommand{\C}{\mathfrak c}
\newcommand{\D}{\mathfrak  d}

\newcommand{\Hd}{\mbf H_d}

\newcommand{\z}{\mathfrak z}

\newcommand{\md}{\mathfrak d}

\newcommand{\m}{\mathfrak m}

\newcommand{\ro}{\mrm{ro}}
\newcommand{\co}{\mrm{co}}

\newcommand{\wt}{\widetilde}

\newcommand{\U}{\mbf U}

\newcommand{\End}{\mrm{End}}

\usepackage{color}
\newcommand{\nc}{\newcommand}
\nc{\redtext}[1]{\textcolor{red}{#1}}
\nc{\bluetext}[1]{\textcolor{blue}{#1}}
\nc{\greentext}[1]{\textcolor{green}{#1}}
\nc{\yl}[1]{\redtext{From yq: #1}}
\nc{\zb}[1]{\redtext{From zb: #1}}

\title[Geometric Schur Duality of two parameter quantum group of type A]{Geometric Schur Duality of two parameter quantum group of type A}

\author{Haitao Ma}
\author{Zongzhu Lin}
\author{Zhu-Jun Zheng}
\address{1. Department of Mathematics, South China University of Technology,  Guangzhou,  China 510641
        }
\email{Zhengzj@scut.edu.cn}
\address{2. Department of Mathematics, Kansas State University, Mahattan, Kansas 66506
        }
\email{zlin@math.ksu.edu}

\date{}

\begin{document}

\begin{abstract}
In this paper, we give an geometric description of the Schur-Weyl duality for two-parameter quantum algebras $U_{v, t}(gl_n)$,  where $U_{v, t}(gl_n)$ is the deformation of $U_v(I, \cdot)$,  the classic Shur-Weyl duality $(U_{r, s}(gl_n), V^{\otimes d}, H_d(r, s))$ can be seen as a corollary of the Shur-Weyl duality $(U_{v, t}(gl_n), V^{\otimes d}, H_d(v, t))$ by using the galois descend approach. we also establish the Shur-Weyl duality between the algebras $\widetilde{U_{v, t}(gl_N)^m}$,  $\widehat{U_{v, t}(gl_N)^m}$ and Heck algebra $H_k(v, t)$.

\end{abstract}

\maketitle

\section{Introduction}
Schur-Weyl duality is a classical method to construct irreducible modules of simple Lie groups out of the fundamental representations \cite{W46},
 The quantum version for the quantum enveloping algebra $U_q(sl_n)$ and the Hecke algebra $H_q(S_m)$ has been one of the pioneering examples [13] in the fervent development of quantum groups.  Two-parameter general linear and special linear quantum groups [21,  8,  4] are certain generalization of the one-parameter Drinfeld-Jimbo quantum groups [7,  12].  The two-parameter quantum groups also had their origin in the quantum inverse scattering method [20] as well as other approaches [14,  6].  So far,  lots of mathematicians had studied the quantum groups and two parameter quantum group.  For example,  geometric Shur-Jimbo duality of type A was studied by Beilinson,  Lusztig and Mcpherson \cite{BLM90}.  And the Shur-like duality of type B/C and D were discovered by
 Bao-Wang \cite{BKLW13}and Fan-Li\cite{FL14}.

 Especially,  Fan and  Li had found another version of two parameter quantum group by the way of perverse sheaves \cite{FL13}.
But the question how the two parameter quantum group $U_{v, t}(gl_n)$ can be seen as the deformation of $U_{v}(gl_n)$ didn't solve in their work.  So it is necessary for us to give the new graded structure on $U_{v}(gl_n)$ such that $U_{v, t}(gl_n)$ can be seen as the deformation of $U_{v}(gl_n)$.

 Fan and Li found two new quantum group $\mathbf{U}$ and $\mathbf{U}^m$,  and  gave the Shur-Weyl duality between them and the Iwahori-hecke algebra of type $D_d$\cite{FL14}.  In our following paper,  similar to the Fan and Li's work,  we will give two new two parameter quantum group $\mathbf{U}_{v, t}$ and $\mathbf{U}^m_{v, t}$.  We can also give the Shur-Weyl duality between them and the two parameter Iwahori-hecke algebra of type $D_d$ through the geometric way.  In order to give the comutiplication of the two new two parameter quantum group $\mathbf{U}_{v, t}$,  $\mathbf{U}^m_{v, t}$ and use the comutiplication structure to give the Shur-Weyl duality algebraically.  That is,
 $$\Delta: \mathbf{U} \rightarrow \widetilde{U_{v, t}(gl_N)^m} \otimes \mathbf{U}, $$
 $$\Delta: \mathbf{U}^m \rightarrow  \widehat{U_{v, t}(gl_N)^m}\otimes \mathbf{U}^m. $$
  So it is reasonable for us to give structure of the new quantum group $\widetilde{U_{v, t}(gl_N)^m}$,  $\widehat{U_{v, t}(gl_N)^m}$ and the Shur-Weyl duality between them and $H_{v, t}(d)$.

  In this work,  at first,  we give a new version of two parameter quantum group $U_{v, t}(gl_n)$,  which is the deformation of $U_{v}(gl_n)$ similar to the approach appear in \cite{FL13}.
 Second,  we would like to give the geometric  realization of  three quantum groups $U_{v, t}(gl_n)$,  $\widetilde{U_{v, t}(gl_N)^m}$,  $\widehat{U_{v, t}(gl_N)^m}$.  At the same time,  we also give the Shur-Weyl duality between algebras $U_{v, t}(gl_n)$,  $\widetilde{U_{v, t}(gl_N)^m}$,   $\widehat{U_{v, t}(gl_N)^m}$ and the
Hecke algebra $H_{v, t}(d)$.  Since the classical two parameter quantum group $U_{r, s}(gl_n)$ is the subalgebra of the new version $U_{v, t}(gl_n)$.  we would like to use the Galois descend approach to understand the two different versions of two parameter quantum groups.  The classical Shur-Weyl duality $(U_{r, s}(gl_n), V^{\otimes d}, H_d(r, s))$ can be seen as a corollary of the Shur-Weyl duality $(U_{v, t}(gl_n), V^{\otimes d}, H_d(v, t))$ by using the galois descend theory.  That is, there exist a Galois group $G$ such that $(U_{v, t}(gl_n)^G, (V^{\otimes d})^{G}, H_d(v, t)^G)$ is also Shur-Weyl duality,  and $U_{v, t}(gl_n)^G \cong (U_{r, s}(gl_n)$,  $H_d(v, t)^G \cong H_d(r, s)$.

\section{Deformation}

\subsection{The algebra $U_{v, t}(gl_n)$}

let  $\Omega = \left(
                             \begin{array}{cccccc}
                               1 & 0 & 0 & \dots & 0 & 0 \\
                               -1 & 1 & 0 & \dots & 0 & 0 \\
                               0 & -1 & 1 & \dots & 0 & 0 \\
                               \vdots & \vdots & \vdots & \vdots & \vdots & \vdots \\
                               0 & 0 & 0 & \dots & 1 & 0 \\
                               0 & 0 & 0 & \dots & -1 & 1 \\
                             \end{array}
                           \right)_{n \times n}
$
is associated to the cartan matrix of type $A_{n}. $
Let $I=\{1, 2, \cdots, n\}$ .
To $\Omega$,  we associate the following three bilinear forms on $\mathbb{Z}^I$.
\begin{eqnarray}
\langle i,  j\rangle &=&
\Omega_{ij},  \quad \hspace{45pt}\forall i,  j\in I.  \label{eq47}\\
\begin{bmatrix} i, j \end{bmatrix}&=& 2\delta_{ij} \Omega_{ii} -\Omega_{ij},  \quad \forall i,  j\in I.  \label{eq48}\\
i\cdot j&=&\langle i,  j\rangle +\langle j, i\rangle,   \quad \forall i,  j\in I.  \label{eq49}
\end{eqnarray}

\begin{Def}

{\it The two-parameter quantum algebra} $U_{v, t}(gl_{n})$ associated to $A_{n - 1}$ is an associative $\mathbb{Q}(v, t)$-algebra with 1 generated by symbols $E_i,  F_i, $ $\forall i\in I$  $A_j^{\pm 1},  B_j^{\pm 1}$, $\forall i\in I' = I \cup \{n\}$ and subject to the following relations.
\begin{eqnarray*}
  (R1)& & A_i^{\pm 1}A^{\pm 1}_j=A^{\pm 1}_jA_i^{\pm 1}, \ \ B^{\pm 1}_iB^{\pm 1}_j=B^{\pm 1}_jB^{\pm 1}_i, \\
     & & A_i^{\pm 1}B^{\pm 1}_j=B^{\pm 1}_jA_i^{\pm 1}, \ \ A_i^{\pm 1}A_i^{\mp 1}=1=B^{\pm 1}_iB^{\mp 1}_i. \\
  (R2)& &A_iE_jA^{-1}_i=v^{\langle i, j\rangle}t^{\langle i, j\rangle}E_j, \ \ B_iE_jB^{-1}_i=v^{-\langle i, j\rangle}t^{\langle i, j\rangle}E_j, \\
     & &A_iF_jA^{-1}_i=v^{-\langle i, j\rangle}t^{- \langle j, i\rangle}F_j, \ \ B_iF_jB^{-1}_i=v^{\langle i, j\rangle}t^{-\langle j, i\rangle}F_j. \\
  (R3)& &E_iF_j-F_j E_i=\delta_{ij}\frac{A_iB_{i+1}-B_iA_{i+1}}{v-v^{-1}}. \\
  (R4)& & \sum_{p+p'=1-2\frac{i\cdot j}{i\cdot i}}(-1)^pt^{-p(p'-2\frac{\langle i, j\rangle}{i\cdot i}+2\frac{\langle j, i\rangle}{i\cdot i})}E_i^{(p')}E_j E_i^{(p)}=0, \quad {\rm if}\ i\not =j, \\
  & &\sum_{p+p'=1-2\frac{i\cdot j}{i\cdot i}}(-1)^pt^{-p(p'-2\frac{\langle i, j\rangle}{i\cdot i}+2\frac{\langle j, i\rangle}{i\cdot i})}F_i^{(p)}F_j F_i^{(p')}=0, \quad {\rm if}\ i\not =j,
\end{eqnarray*}
where $E^{(p)}_i=\frac{E_i^p}{[p]^!_{v_i, t_i}}$. $\langle j,  n\rangle = 0$,
$ \langle n,  j\rangle =\left\{\begin{array}{ll}
-1 & \text{if $j = n-1;$ } \\[.05in]
0
 & \text{else
 }.
 \end{array}   \right.
$,  $j \in I$.
\end{Def}

 The algebra $U_{v, t}(gl_n)$ has a Hopf algebra structure with the comultiplication $\Delta$,  the counit $\varepsilon$ and the antipode $S$ given as follows.

  $$\begin{array}{llll}
   &\Delta(A_i^{\pm 1})=A_i^{\pm 1} \otimes A_i^{\pm 1}, &\Delta(B^{\pm 1}_i)=B^{\pm 1}_i \otimes B^{\pm 1}_i,  &\vspace{4pt}\\
   &\Delta(E_i)=E_i\otimes A_iB_{i+1} + 1 \otimes E_i, & \Delta(F_i)=F_i \otimes 1+ B_iA_{i+1} \otimes F_i ,  & \vspace{4pt}\\
  &\varepsilon(A_i^{\pm 1})=\varepsilon(B'^{\pm 1}_i)=1, & \varepsilon(E_i)=\varepsilon(F_i)=0, & S(A_i^{\pm 1})=A_i^{\mp 1}, \vspace{4pt}\\
  & S(B^{\pm 1}_i)=B^{\mp 1}_i, &
  S(E_i)=-E_iB_{i}A_{i+1}, &S(F_i)=-A_{i}B_{i + 1}F_i.
  \end{array}$$

 The algebra $U_{v, t}(gl_n)$  admits a $\mathbb{Z}^{I'} \times \mathbb{Z}^{I'}$-grading  by  defining the degrees of generators  as follows.
\begin{eqnarray*}
& &deg(E_i)=(i, 0), \quad  \ deg(F_i)=(0, i), \\
\end{eqnarray*}
$$ deg(A_{j})=deg(B_{j}) =\left\{\begin{array}{ll}
(\sum\limits_{k = j}^{n}(-1)^k k, \sum\limits_{k = j}^{n}(-1)^k k)& \text{if j is even,  } \\[.05in]
(\sum\limits_{k = j}^{n}(-1)^{k+1} k, \sum_{k = j}^{n}(-1)^{k+1} k)& \text{if j is odd. }
 \end{array}   \right.
$$

 We can define a bilinear form on $\mathbb{Z}^{I'} \times \mathbb{Z}^{I'}$ by
\begin{eqnarray*}
[ \gamma,  \eta ]'=[\gamma_2,  \eta_2]-[\gamma_1,  \eta_1]
\end{eqnarray*}
for any $\gamma=(\gamma_1,  \gamma_2),  \eta=(\eta_1, \eta_2) \in \mathbb{Z}^{I'} \times \mathbb{Z}^{I'}$.
Then on $U_{v, t}(gl_{n})$,  we can define a new multiplication $''\ast"$ by
\begin{equation}\label{eq63}
  x \ast y=t^{-[ |x|, \ |y|]'}xy,
\end{equation}
for any homogenous elements $x,  y \in U_{v, t}(gl_n)$.
Since $[, ]'$ is a bilinear form,
$(U_{v, t}(gl_n),  *)$ is an associative algebra over $\mathbb{Q}(v, t)$.
We define a multiplication,  denoted by $``*"$,  on $U_{v, t}(gl_n)\otimes U_{v, t}(gl_n)$ by
\begin{equation}\label{eq72}
(x\otimes y)*(x'\otimes y')=x*x' \otimes y*y'.
\end{equation}
This gives a new algebra structure on $U_{v, t}(gl_n)\otimes U_{v, t}(gl_n)$.
  $(U_{v, t}(gl_n),  *)$ has a Hopf algebra structure with the comultiplication $\Delta^*$,  the counit $\varepsilon^*$ and the antipode $S^*$.  The image of generators $E_i,  F_i, A_i$ and $B_i^{-1}$ under the map $\Delta^*$ (resp.  $\varepsilon^*$ and $S^*$) are the same as the ones under the map $\Delta$ (resp.  $\varepsilon$ and $S$) defined above.
\begin{lem}
Under the new multiplication $``\ast"$,   the defining relations of $U_{v, t}(gl_n)$  can be rewritten as follows.
\begin{eqnarray*}
  (R^*1)& & A_i^{\pm 1}\ast A^{\pm 1}_j=A^{\pm 1}_j\ast A_i^{\pm 1}, \ \ B^{\pm 1}_i\ast B^{\pm 1}_j=B^{\pm 1}_j\ast B^{\pm 1}_i, \\
     & & A_i^{\pm 1}\ast B^{\pm 1}_j=B^{\pm 1}_j\ast A_i^{\pm 1}, \ \ A_i^{\pm 1}\ast A_i^{\mp 1}=1=B^{\pm 1}_i\ast B^{\mp 1}_i. \\
  (R^*2)& &A_i\ast E_j\ast A^{-1}_i=v^{\langle i, j\rangle}E_j, \ \ \ \ \ \  B_i\ast E_j\ast B^{-1}_i=v^{-\langle i, j\rangle}E_j, \\
     & &A_i\ast F_j\ast A^{-1}_i=v^{-\langle i, j\rangle}F_j, \ \ \ \ \ B_i\ast F_j\ast B^{-1}_i=v^{\langle i, j\rangle}F_j. \\
  (R^*3)& &E_i\ast F_j-F_j\ast E_i=\delta_{ij}\frac{A_i\ast B_{i+1} - B_{i}\ast A_{i+1}}{v-v^{-1}}, \quad \forall i,  j\in I. \\
  (R^*4)& & \sum_{p+p'=1-a_{ij}}(-1)^p\begin{bmatrix} 1-a_{ij}\\ p \end{bmatrix}_{v}E_i^{*p}\ast E_j\ast E_i^{*p'}=0, \quad {\rm if}\ i\not =j, \\
  & &\sum_{p+p'=1-a_{ij}}(-1)^p\begin{bmatrix} 1-a_{ij}\\ p \end{bmatrix}_{v}F_i^{*p}\ast F_j\ast F_i^{*p'}=0\quad {\rm if}\ i\not =j,
\end{eqnarray*}
 where $a_{ij}=2\frac{i\cdot j}{i\cdot i}$ and $E_i^{*p}=E_i*E_i*\cdots *E_i$ for $p$ copies.  We notice that these relations are the specialization of (R1)-(R4) at $t=1$.
\end{lem}

\begin{proof}
The relation  R3, R4 agrees with the one in  ~\cite[4. 2]{FL13},  whose  proof is also the same as the one for type-$\mbf A$ case.  Next we show $R^*2$.
$$A_i\ast E_j = t^{-[ |A_i|, \ |E_j|]'}A_iE_j = t^{-[ |A_i|, \ |E_j|]'}v^{\langle i, j\rangle}t^{\langle i, j\rangle}E_jA_i = t^{[ |E_j|, \|A_i| ]'-[ |A_i|, \ |E_j|]'}v^{\langle i, j\rangle}t^{\langle i, j\rangle}E_j\ast A_i$$
and
$$[ |E_j|, \|A_i| ]'-[ |A_i|, \ |E_j|]' = (\langle j, i \rangle- \langle i, j \rangle)-(\langle j, i+1 \rangle- \langle i+1, j \rangle) + \cdots +(-1)^{n-i}(\langle j, n \rangle - \langle n, j \rangle). $$

\begin{equation*}
[ |E_j|, \|A_i| ]'-[ |A_i|, \ |E_j|]=\left\{\begin{array}{ll}
 = \langle i+1,  j \rangle = 1 = - \langle i, j \rangle & {\rm if}\ i = j, \vspace{5pt}\\
 = (\langle j, i \rangle- \langle i, j \rangle) - (\langle j, i+1 \rangle- \langle i+1, j \rangle) = 1 = - \langle i, j \rangle & {\rm if} \ i = j + 1,  \vspace{5pt}\\
 =  0 = - \langle i, j \rangle & {\rm if}\ i - j > 1,  \vspace{5pt}\\
= \langle j, i \rangle- \langle i + 1, j \rangle  = 0 = - \langle i, j \rangle & {\rm if}\ j = i + 1, \vspace{5pt}\\
= -\langle j - 1, j \rangle +  \langle j, j-1 \rangle  = 0 = - \langle i, j \rangle  & {\rm if}\ j - i > 1.
  \end{array}\right.
\end{equation*}

Therefore,
$$A_i\ast E_j = v^{\langle i, j\rangle}E_j\ast A_i. $$

All other identity in R2 can be shown similarly.
\end{proof}

  The one-parameter quantum algebra $U_{v}(I,  \cdot)$ associated to  $(I,  \cdot)$ is defined as the associative $\mathbb{Q}(v)$-algebra with 1 generated by symbols $E_i,  F_i,  A_i^{\pm 1}$,  $ B_i^{\pm 1},  \forall i\in I$ and  subject to relations (R*1)-(R*4).
  $U_{v}(I,  \cdot)$ has a Hopf algebra structure with the comultiplication $\Delta_1$,  the counit $\varepsilon_1$ and the antipode $S_1$.   The image of generators $E_i,  F_i, A_a,  B_a$ under the map $\Delta_1$ (resp.  $\varepsilon_1$ and $S_1$) are the same as the ones under the map $\Delta$ (resp.  $\varepsilon$ and $S$) defined above.

Let $U_{v, t}(I,  \cdot):=U_{v}(I,  \cdot)\otimes_{\mathbb{Q}(v)}\mathbb{Q}(v, t)$.  The Hopf algebra structure on $U_{v}(I,  \cdot)$ can be naturally extended to $U_{v, t}(I,  \cdot)$.
From the above analysis,  we have the following theorem.
\begin{thm}\label{thm14}
  If $(I,  \cdot)$ is the Cartan datum associated to $\Omega_{n}$,  then there is a Hopf-algebra isomorphism
 $$(U_{v, t}(gl_n),  \ast,  \Delta^*,  \varepsilon^*,  S^*)\simeq (U_{v, t}(I, \cdot),  \cdot,  \Delta_1,  \varepsilon_1,  S_1), $$
 sending the generators in $U_{v, t}$ to the respective generators in $U_{v, t}(I, \cdot)$.
\end{thm}

\section{A geometric setting}

\subsection{Preliminary}

Let $\mbb F_q$ be a finite field of $q$ elements and of odd characteristic.
 $d$ is a fixed positive integer,

$n$ is a positive integer ,
We fix a  vector space  $\mbb F_q^d$ .   Consider the following sets.
\begin{itemize}
\item The set $\mscr X$ of $n$-step flags $V=( V_i)_{ 0\leq i\leq n}$
          in $\mbb F_q^d$ such that  $V_0 = 0$, $V_i\subseteq V_{i+1}$.

\item The set $\mscr Y$  of complete flags  $F= (F_i)_{0\leq i\leq d} $ in $\mbb F_q^d$
          such that $F_i\subset F_{i+1}$,  $|F_i|=i$ .
\end{itemize}
where we write $|F_i|$ for the dimension of $F_i$.

Let $G=\mrm{GL}(V)$ .
Then $G$ acts naturally on sets $\mscr X$ and $\mscr Y$.  Moreoever,
$G$ acts  transitively on $\mscr Y$.
Let $G$ act diagonally on the product $\mscr X\times \mscr X$ (resp.  $\mscr X\times \mscr Y$ and $\mscr Y\times \mscr Y$).
Set
\begin{equation}\label{eq32}
  \mcal A=\mbb Z[v^{\pm1},  t^{\pm1}].
\end{equation}
Let
\begin{align}
\label{S(X)}
\mcal S_{\mscr X}=\mcal A_G(\mscr X\times\mscr X)
\end{align}
be the set of all $\mcal A$-valued $G$-invariant functions on $\mscr X\times \mscr X$.
Clearly,  the set $\mcal S_{\mscr X}$ is a  free $\mcal A$-module.
Moreover,   $\mcal S_{\mscr X}$ admits  an associative  $\mcal A$-algebra structure `$*$' under a standard  convolution product as discussed in ~\cite[2. 3]{BKLW13}.  In particular,  when $v$ is specialized to $\sqrt q$,  we have
\begin{equation}
  \label{eq30}
  f * g(V,  V')=\sum_{V''\in \mscr X}f(V,  V'')g(V'', V'),  \quad \forall\ V, V'\in \mscr X.
\end{equation}
Similarly,  we define  the free $\mcal A$-modules
\begin{equation}\label{V}
\mcal V=\mcal A_G(\mscr X\times\mscr Y)
\quad
\mbox{and}
\quad
\mcal H_{\mscr Y}=\mcal A_G(\mscr Y\times \mscr Y).
\end{equation}
A similar  convolution product gives an associative algebra structure on $\mcal H_{\mscr Y}$ and
 a left $\mcal S_{\mscr X}$-action   and a right $\mcal H_{\mscr Y}$-action on $\mcal V$.
Moreover,  these two actions commute and hence we have the following $\mcal A$-algebra homomorphisms.
\[
\mcal S_{\mscr X} \to \End_{\mcal H_{\mscr Y}} (\mcal V)
\quad\mbox{and}\quad
\mcal H_{\mscr Y} \to \End_{\mcal S_{\mscr X}} (\mcal V).
\]
Similar to  ~\cite[Theorem 2. 1]{P09},  we have the following double centralizer property.

\begin{lem}
\label{eq34}
$\End_{\mcal H_{\mscr Y}}(\mcal V)\simeq \mcal S_{\mscr X}$ and
$\End_{\mcal S_{\mscr X}}(\mcal V) \simeq \mcal H_{\mscr Y}$,  if
$ n\geq d. $
\end{lem}

We note that the  result in ~\cite[Theorem 2. 1]{P09} is obtained over  the field $\mbb C$ of complex numbers,
but the  proof can be adapted  to  our setting   over the ring  $\mcal A$.

We shall give a description of the $G$-orbits on  $\mscr X\times \mscr X$,  $\mscr X\times \mscr Y$  and  $\mscr Y\times \mscr Y$.
We start by introducing the following notations associated to a matrix
$M=(m_{ij})_{1\leq i,  j \leq c}$.
\begin{align} \label{ro-co}
\begin{split}
\ro (M) & = \left (\sum_{j=1}^{c}  m_{ij} \right )_{1\leq i\leq c},  \\
\co (M) & = \left (\sum_{i=1}^{c} m_{ij} \right )_{1\leq j \leq c}.
\end{split}
\end{align}
We also write $\ro(M)_i$ and $\co (M)_j$ for the $i$-th and $j$-th component of the row vectors of $\ro (M)$ and $\co(M)$,  respectively.

For any pair $(V,  V')$ of flags in $\mscr X$,  we can assign an $n$ by $n$ matrix whose $(i, j)$-entry equal to
$\dim \frac{V_{i-1}+ V_i\cap V_j'}{V_{i-1} + V_i\cap V_{j-1}'}$.

\begin{align}
G \backslash \mscr X\times \mscr X \simeq  \Theta_d,
\end{align}
where
$\Theta_d$ is the set of all matrices $\Theta_d $ in $\mbox{Mat}_{n \times n} (\mbb N)$ such that
$\sum_{i, j}(\Theta_d)_{i, j} = d$

A similar assignment yields two bijection
\begin{align}\label{eq44}
G \backslash \mscr X\times \mscr Y \simeq \Pi , \\
G \backslash \mscr Y\times \mscr Y \simeq  \Sigma,
\end{align}
where the set $\Pi$ consists of all matrices $B=(b_{ij})$ in $\mrm{Mat}_{n\times d} (\mbb N)$ subject to
\[
 \co (B)_j =1, \  \forall \  j\in [1,  d].
 \]
and $\Sigma$ is the set of all matrices $\sigma \equiv (\sigma_{ij})$ in $\mbox{Mat}_{d\times d} (\mbb N)$ such that
\begin{align*}
 \ro (\sigma)_i = 1,  \quad  \ro (\sigma)_j =1.
\end{align*}
Moreover,  we have
\begin{equation}
 \begin{split}
\# \Sigma = d! \quad {\rm and}\quad
\# \Pi = n^d.
\end{split}
\end{equation}

\section{Calculus of the algebra $\mcal S$  and $\mcal H_{\mscr Y}$}
\label{secS_D}

Recall from the previous section  that $\mcal S_{\mscr X}$ is the convolution algebra on $\mscr X\times \mscr X$
defined in (\ref{S(X)}).
For simplicity,  we shall denote $\mcal S$ instead of $\mcal S_{\mscr X}$.
In this section,  we determine  the generators for $\mcal S$ and the associated  multiplication formula.  We also will $\mcal H_{\mscr Y}$ action  on $\mcal V$.

\subsection{Defining relations of $\mcal S$}

 For any $i\in [1,  n-1]$,  $a\in [1,  n]$,   set
\begin{align}
\begin{split}
E_i (V,  V') &=
\begin{cases}
v^{-|V'_i/V'_{i-1}|}t^{-|V_i/V_{i-1}|},  &\mbox{if}\; V_i\overset{1}{\supset} V_i',  V_j=V_{j'}, \forall j\in [1, n]\backslash \{i\}; \\
0,  &\mbox{otherwise}.
\end{cases}
\\
F_i (V,  V') &=
\begin{cases}
v^{-|V'_{i+1}/V'_i|}t^{|V'_{i+1}/V'_i|},  &\mbox{if}\; V_i\overset{1}{\subset} V_i',  V_j=V_{j'}, \forall j\in [1, n]\backslash \{i\}; \\
0,  &\mbox{otherwise}.
\end{cases}
\\
A_a^{\pm 1} (V,  V') & =
\begin{cases}
v^{\pm |V_a'/V_{a-1}'|}t^{\pm |V_a'/V_{a-1}'|} ,  &\mbox{if}\; V=V';\\
0,  & \mbox{otherwise}.
\end{cases}
\\
B_a^{\pm 1} (V,  V') & =
\begin{cases}
v^{\mp |V_{a}'/V_{a-1}'|}t^{\pm |V_{a}'/V_{a-1}'|} ,  &\mbox{if}\; V=V';\\
0,  & \mbox{otherwise}.
\end{cases}
\end{split}
\end{align}

It is clear that these functions are elements in $\mcal S$.

\begin{prop}\label{prop3}
The functions $E_i,  F_i$,   $A_a^{\pm 1}$ and $B_a^{\pm 1}$  in $\mcal S$,  for any $i\in [1, n-1]$,  $a\in [1, n]$,    satisfy the following relations.
\begin{eqnarray*}
  (R1)& & A_i^{\pm 1}A^{\pm 1}_j=A^{\pm 1}_jA_i^{\pm 1}, \ \ B^{\pm 1}_iB^{\pm 1}_j=B^{\pm 1}_jB^{\pm 1}_i, \\
     & & A_i^{\pm 1}B^{\pm 1}_j=B^{\pm 1}_jA_i^{\pm 1}, \ \ A_i^{\pm 1}A_i^{\mp 1}=1=B^{\pm 1}_iB^{\mp 1}_i. \\
  (R2)& &A_iE_jA^{-1}_i=v^{\langle i, j\rangle}t^{\langle i, j\rangle}E_j, \ \ B_iE_jB^{-1}_i=v^{-\langle i, j\rangle}t^{\langle i, j\rangle}E_j, \\
     & &A_iF_jA^{-1}_i=v^{-\langle i, j\rangle}t^{- \langle j, i\rangle}F_j, \ \ B_iF_jB^{-1}_i=v^{\langle i, j\rangle}t^{-\langle j, i\rangle}F_j. \\
  (R3)& &E_iF_j-F_j E_i=\delta_{ij}\frac{A_iB_{i+1}-B_iA_{i+1}}{v-v^{-1}}. \\
  (R4)& & \sum_{p+p'=1-2\frac{i\cdot j}{i\cdot i}}(-1)^pt^{-p(p'-2\frac{\langle i, j\rangle}{i\cdot i}+2\frac{\langle j, i\rangle}{i\cdot i})}E_i^{(p')}E_j E_i^{(p)}=0, \quad {\rm if}\ i\not =j, \\
  & &\sum_{p+p'=1-2\frac{i\cdot j}{i\cdot i}}(-1)^pt^{-p(p'-2\frac{\langle i, j\rangle}{i\cdot i}+2\frac{\langle j, i\rangle}{i\cdot i})}F_i^{(p)}F_j F_i^{(p')}=0, \quad {\rm if}\ i\not =j, \\
  (R5)&&\prod_{i=1}^n A_i=v^{d}t^{d}, \quad \prod_{i=1}^n B_i=v^{-d}t^{d};\\
  (R6)&& \prod_{l=0}^d(A_j-v^{l}t^{l})=0, \quad \prod_{l=0}^d(B_j-v^{-l}t^{l})=0 \forall j\in [1, n]. \\
  (R7)&& {E_i}^{d+1} = 0,  {F_i}^{d+1} = 0.
\end{eqnarray*}
\end{prop}

\begin{proof}
The proofs of the  identities of R1, R7 are straightforward.
Let $\lambda'_i=|V'_i/V'_{i-1}|$.
We show the first identity in R2.  we have
\begin{equation*}
(A_iE_{j})(V, V')=\left\{\begin{array}{ll}
 v^{\lambda_i' - \lambda_j' - 1}t^{\lambda_i' + \lambda_j'}& {\rm if}\ V_j\overset{1}{\supset} V_j' {\rm and} \ i=j + 1, \vspace{5pt}\\
 v^{\lambda_i' - \lambda_j' + 1}t^{\lambda_i' + \lambda_j' + 2}& {\rm if}\ V_j\overset{1}{\supset} V_j' {\rm and} \ i=j,  \vspace{5pt}\\
 v^{\lambda_i' - \lambda_j'}t^{\lambda_i' + \lambda_j' + 1} & {\rm if}\ V_j\overset{1}{\supset} V_j' {\rm and} \ i\neq j,  j+1,  \vspace{5pt}\\
 0 & {\rm otherwise}.
  \end{array}\right.
\end{equation*}
\begin{equation*}
(A_iE_{j})(V, V')=\left\{\begin{array}{ll}
 v^{\lambda_i' - \lambda_j'}t^{\lambda_i'+ \lambda_j'+1}& {\rm if}\ V_j\overset{1}{\supset} V_j' , \vspace{5pt}\\
 0 & {\rm otherwise}.
  \end{array}\right.
\end{equation*}

That is,  $A_iE_jA^{-1}_i(V, V')=v^{\langle i, j\rangle}t^{\langle i, j\rangle}E_j(V, V'). $ All other identities can be shown similarly.

we show the identity in R3. By a direct calculation.
We have
\begin{equation*}
(E_iF_{j}-F_jE_i)(V, V')=\left\{\begin{array}{ll}
 \frac{v^{\lambda_i' - \lambda_{i + 1}' }t^{\lambda_i' + \lambda_{i + 1}'}-v^{\lambda_{i+1}'- \lambda_{i}'}t^{\lambda_i' +\lambda_{i + 1}'}}{v-v^{-1}}& {\rm if}\ V=V'{\rm and} i=j, \vspace{5pt}\\
  0&{\rm otherwise}.
  \end{array}\right.
\end{equation*}
It is easy to check that the right hand side is equal to $\delta_{ij}\frac{A_iB_{i+1}-B_iA_{i+1}}{v-v^{-1}}(V, V')$.

At last,
We now show the first identity in R4.
By a direct calculation,  we have
 \begin{equation*}
    E_i^2E_{i + 1}(V, V')=\left\{\begin{array}{l}
      (v^2+1)v^{-2\lambda'_i-\lambda_{i+1}'- 1}t^{2\lambda'_i + \lambda_{i+1}'+ 4} , \quad {\rm if}\  V_i\overset{2}{\supset} V_i' {\rm and}\ V_{i + 1}\overset{1}{\supset} V_{i + 1}',  \\
      0,  \hspace{159pt} {\rm otherwise}.
    \end{array}
  \right.
  \end{equation*}
  \begin{equation*}
    E_iE_{i + 1}E_i(V, V')=\left\{\begin{array}{l}
      (v^2+1)v^{-2\lambda'_i-\lambda_{i+1}'}t^{2\lambda'_i + \lambda_{i+1}' + 3}, \quad {\rm if}\  V_i\overset{2}{\supset} V_i' {\rm and}\ V_{i + 1}\overset{1}{\supset} V_{i + 1}', \\
      0, \hspace{148pt}{\rm otherwise}.
    \end{array}
  \right.
  \end{equation*}
  \begin{equation*}
    E_{i + 1} E_i^2(V, V')=\left\{\begin{array}{ll}
      (v^2+1)v^{-2\lambda'_i-\lambda_{i+1}' + 1}t^{2\lambda'_i + \lambda_{i+1}' + 2},  &{\rm if}\  V_i\overset{2}{\supset} V_i' {\rm and}\ V_{i + 1}\overset{1}{\supset} V_{i + 1}', \\
      0,  & {\rm otherwise}.
    \end{array}
  \right.
\end{equation*}
The first identity in R4 follows.  By the same way,  the other three identities can be shown directly.

Let's prove the first identity in R5,  we have

  \begin{equation*}
  \prod_{i=1}^n A_i(V, V')=\left\{\begin{array}{ll}
     v^{\lambda_1'+ \dots +\lambda_n'} t^{\lambda_1'+ \dots +\lambda_n'},  &{\rm if}\  V = V', \\
      0,  & {\rm otherwise}.
    \end{array}
  \right.
\end{equation*}
Since $\lambda_1'+ \dots +\lambda_n' = d$,  the first identities follows.  The other identities can be shown similarly.

At last,  let's prove the first identity in R6,  we have

  \begin{equation*}
  \prod_{l=0}^d(A_j-v^{l}t^{l})(V,  V')=\left\{\begin{array}{ll}
     (v^{\lambda_j'} t^{\lambda_j'} - 1)(v^{\lambda_j'} t^{\lambda_j'} - vt) \dots (v^{\lambda_j'} t^{\lambda_j'} - v^dt^d),  &{\rm if}\  V = V', \\
      0,  & {\rm otherwise}.
    \end{array}
  \right.
\end{equation*}
Since $0 \leq \lambda_j' \leq d $,  the first identities follows.  the other identities can be shown similarly.
\end{proof}

\subsection{Multiplication formulas in $\mcal S$}\label{sec4. 3}

For any $n\in \mbb Z,   k\in \mbb N$,  set
$$(n)_v=\frac{v^{2n}-1}{v^2-1},
\quad {\rm and}\quad
\begin{pmatrix}
  n\\k
\end{pmatrix}_{\!\!\!{}v}= \prod_{i=1}^k\frac{(n+1-i)_v}{(i)_v}.
$$
Let $E_{ij}$ is the $n\times n$ matrix whose $(i, j)$-entry is 1 and all other entries are 0.
Let $e_{\A }$  be the characteristic function of the $G$-orbit corresponding to $\A\in \theta_d$.
It is clear that  the set $\{e_{\A  }| \A \in \theta_d\}$ forms a basis of $\mcal S$.

We assume that the ground field is an algebraic closure $\overline{\mbb F}_q$ of $\mbb F_q$ when we talk about the dimension of a $G$-orbit or its stabilizer.
Set
\[
d(\A)={\rm dim} \ \mcal O_{\A}
\quad\mbox{and}\quad
r(\A)={\rm dim} \ \mcal O_{\B},  \quad \forall \A\in \theta_d \ {\rm or} \  \Pi,
\]
where $\B=(b_{ij})$ is the diagonal matrix  such that $b_{ii}=\sum_ka_{ik}$.
Denote by ${\rm C}_{G}(V, V')$ the stabilizer of $(V, V')$ in $G$.

\begin{lem}
\label{dimension1}
 If $ \A \in \Pi$ , We have
\begin{equation*}
\begin{split}
{\rm dim}\ {\rm C}_{G}(V, V')
& = \sum_{i\geq k,  j\geq l}a_{ij}a_{kl},
\quad  {\rm if}\ (V, V') \in \mcal O_{\A}, \\
{\rm dim}\ \mcal O_{\A}
&=\sum_{i<k\ {\rm or}\ j<l}a_{ij}a_{kl}, \\
d(\A)-r(\A)
&= \sum_{i\geq k,  j<l}a_{ij}a_{kl}.
\end{split}
\end{equation*}
\end{lem}

\begin{proof}
The proof is similar with \cite{BLM90},  The only difference we consider is that $ \A \in \Pi$ should be the $n \times d$ matrix.  We can find the subspace $Z_{ij}$ of $V$ such that$V_a = \oplus_{i \leq a;j}Z_{ij}$ for all $a$,  $V_{b}' = \oplus_{j \leq b;i}Z_{ij}$ for all $b$.  $V = \oplus_{ij}Z_{ij}$.  Consider $T \in End(V)$,  $T$ is determined by a family of linear maps $T_{ijkl} : Z_{ij} \rightarrow Z_{kl}$.  If $T |_{V_{a}} = V_a,  T |_{V_{b}'} = V_{b}'$,  one can obtain that if $T_{ijkl} \neq 0$,  then $i \geq k,  j \geq l$.  So we have $dim C_G(V, V') = \sum\limits_{i \geq k,  j \geq l}a_{ij}a_{kl}$,  $dim \mcal O_{\A} = dim GL(V) - dim C_G(V, V') = \sum\limits_{i<k\ {\rm or}\ j<l}a_{ij}a_{kl}$.  Since $r(\A) = dim(V,  V)$,  we have $d(\A)-r(\A) =\sum\limits_{i<k\ {\rm or}\ j<l}a_{ij}a_{kl} - \sum\limits_{i<k}a_{ij}a_{kl} =  \sum_{i\geq k,  j<l}a_{ij}a_{kl}. $
\end{proof}

For any $\A\in \theta_{d}, \Pi$,  let
$$\{\A\}=v^{-(d(\A)-r(\A))}t^{(d(\A)-r(\A))}e_{\A}. $$
We define a bar involution `$-$' on $\mcal A$ by $\bar v= v^{-1}$.

\begin{prop}\label{cor5}
Suppose that $\A$,  $\B$,  $\C  \in \Theta_d$,  $h\in [1,  n-1]$ and $r\in \mbb N$.

$(a)$ If ${\rm co}(\B)={\rm ro}(\A)$,
and  $\B-rE_{h, h+1}$  is diagonal,   then we have
\begin{align}\label{eq22}
\{\B\} *  \{\A\}
& =\sum_{t: \sum_{u=1}^n t_u=r} v^{\beta(t)}t^{\alpha(t)}\prod_{u=1}^n
\overline{\begin{pmatrix}a_{hu}+t_u\\ t_u \end{pmatrix}}_{\!\!\!{}v}\
\{\A_{t}\},  \ \mbox{where} \\
\alpha(t)
& = \sum_{j\geq l} a_{hj} t_l + \sum_{j>l} a_{h+1,  j} t_l - \sum_{j<l} t_j t_l
  \nonumber\\
\beta(t)
& = \sum_{j\geq l} a_{hj} t_l - \sum_{j>l} a_{h+1,  j} t_l + \sum_{j<l} t_j t_l
  \nonumber\\
\A_t
& =A+\sum_{u=1}^nt_u(E_{hu}-E_{h+1, u}) \in \theta_{d}.  \nonumber
\end{align}

$(b)$  If  ${\rm co}(\C)={\rm ro}(\A)$and  $\C-rE_{h+1, h}$ is diagonal,   then
\begin{align}\label{eq24}
\{\C\} * \{\A\}
& =\sum_{t: \sum_{u=1}^n t_u=r}v^{\beta'(t)}t^{\alpha'(t)}\prod_{u=1}^n
\overline{\begin{pmatrix} a_{h+1, u}+t_u\\ t_u\end{pmatrix}}_{\!\!\!{}v}
\ \{\A(h,  t)\},  \ \mbox{where}\\
\alpha'(t)
& =\sum_{j\leq l} a_{h+1, j} t_l + \sum_{j<l} a_{hj} t_l - \sum_{j<l} t_j t_l,   \nonumber\\
\beta'(t)
& =\sum_{j\leq l} a_{h+1, j} t_l - \sum_{j<l} a_{hj} t_l + \sum_{j<l} t_j t_l,   \nonumber\\
\A(h,  t)
&=A-\sum_{u=1}^nt_u(E_{hu}-E_{h+1, u}) \in \theta_{ d}.  \nonumber
\end{align}

\end{prop}
\begin{proof}
In order to give the proof of $(a)$,  We only need to proof the formula $\A(t)$.  By the direct computation.
$$d(\B) - r(\B) = \sum\limits_{j, u}a_{h, j}t_u, $$
$$d(\A) - r(\A) = \sum\limits_{i \geq k, j < l}a_{ij}a_{kl}, $$
$$d(\A_t) - r(\A_t) = \sum\limits_{i \geq k,  j <l}a_{ij}a_{kl} + \sum\limits_{j < u}a_{hj}t_u - \sum\limits_{l > u}a_{h+1, l}t_u + \sum\limits_{u < u'}t_ut_{u'}. $$
Then,
$$\alpha(t) = d(\B) - r(\B) + d(\A) - r(\A) - (d(\A_t) - r(\A_t)) = \sum_{j\geq l} a_{hj} t_l + \sum_{j>l} a_{h+1,  j} t_l - \sum_{j<l} t_j t_l. $$

Similarly,  we can obtain the proposition of $(b)$.
\end{proof}

\subsection{$\mcal S$-action on $\mcal V$}

A degenerate version of Proposition\ref{cor5}  gives us  an explicit description of  the $\mcal S$-action on $\mcal V=\mcal A_G(\mscr X\times \mscr Y)$ as follows.
For any $r_j\in [1, n]$,  we denote $\check{r}_j=r_j+1$ and $\hat{r}_j=r_j-1$.

\begin{cor}\label{cor9}
  For any $1\leq i\leq n - 1,  1\leq a\leq n - 1$,  we have

  \begin{align}
E_i *  \{e_{r_1\cdots r_d}\}
&=v^{\sum_{j>p}\delta_{i,  r_j} - \delta_{i + 1,  r_j}}t^{1+\sum_{j>p}\delta_{i,  r_j} + \delta_{i + 1,  r_j}} \
\{e_{r_1\cdots, \hat{r}_p\cdots r_{d}}\},  \nonumber\\
F_i *\{ e_{r_1\cdots r_d}\}
&=\sum_{1\leq p\leq d:r_p=i}
v^{\sum_{j<p}\delta_{i+1,  r_j} - \delta_{i,  r_j}}t^{\sum_{j<p}\delta_{i,  r_j} + \delta_{i + 1,  r_j}}
\{e_{r_1\cdots, \check{r}_p\cdots r_{d}}\},  \nonumber \\
A_a^{\pm 1} *  \{e_{r_1\cdots r_d}\}
& = v^{\pm \sum_{1\leq j \leq d} \delta_{a,  r_j}} t^{\pm \sum_{1\leq j \leq d} \delta_{a,  r_j}}  \{e_{r_1\cdots r_d}\} \quad {\rm and}\nonumber\\
B_a^{\pm 1} *  \{e_{r_1\cdots r_d}\}
& = v^{\mp \sum_{1\leq j \leq d} \delta_{a,  r_j}} t^{\pm \sum_{1\leq j \leq d} \delta_{a,  r_j}}  \{e_{r_1\cdots r_d}\} \quad {\rm and}\nonumber
\end{align}
\end{cor}

\begin{proof}
The first two identities follow directly  from Proposition \ref{cor5}.
The last two  identities are straightforward.
\end{proof}

\subsection{$\mcal H_{\mscr Y}$-action on $\mcal V$}

\begin{Def} The two parameter Iwahori-Hecke algebra $\Hd(v, t)$ of type $\mbf A_d$ is a unital associative algebra over $\mathbb Q(v, t)$
generated by $T_i$ for $i\in [1,  d - 1]$ and subject to the following relations.
\begin{equation*}
\begin{split}
& T_i^2   =(vt-v^{-1}t) T_i + t^2 ,  \quad 1\leq i \leq d - 1, \\
& T_j T_{j+1} T_j  = T_{j+1} T_j T_{j+1}, \; \,   1\leq j\leq d-2, \\
& T_i T_j  =T_j T_i, \ \hspace{1.9cm} |i - j|>1.
\end{split}
\end{equation*}
\end{Def}

We shall provide an explicit description of the action of $\mcal H_{\mscr Y}$ on $\mcal V$.
For any $1\leq j\leq d-1$,  we define  a function $T_j$ in $\mcal H_{\mscr Y}$  by
\begin{align*}
T_j (F,  F') &=
\begin{cases}
v^{-1}t,  & \mbox{if} \; F_i = F_i' \ \forall i\in [1,  d]\backslash  \{j\},  F_j\neq F_j';\\
0,  & \mbox{otherwise}.
\end{cases}
\end{align*}

\begin{lem} \label{Geometric-Hecke}
 The assignment of sending the  functions $T_j$,   for $1\leq j\leq  d-1$,  in  the algebra $\mcal H_{\mscr Y}$ to the generators of
 $\Hd$ in the same notations is an isomorphism.
\end{lem}

Given $B=(b_{ij})\in \Pi$,  let $r_c$ be the unique number in $[1, n]$ such that $b_{r_c,  c} =1$ for each $c\in [1,  d]$.
The correspondence $B\mapsto \tilde{\mbf r}= (r_1, \cdots,  r_d)$ defines a bijection between $\Pi$ and the set of all sequences
$(r_1, \cdots,  r_d)$ .
Denote by $e_{r_1 \dots r_d}$ the characteristic function of the $G$-orbit corresponding to the matrix $B$ in $\mcal V$.
It is clear that the collection of these characteristic functions provides a basis for $\mcal V$.

\begin{lem}
\label{H-action}
The action of $\mcal H_{\mscr Y}$ on $\mcal V$ is described as follows.
For $1\leq j\leq d-1$,  we have
\begin{eqnarray}
& \{e_{r_1 \dots r_d }\} T_j =
\begin{cases}
\{e_{r_1\dots r_{j-1} r_{j+1} r_j\dots r_d}\} ,   &  r_j < r_{j+1};\\
vt \{e_{r_1 \dots r_d}\},  &  r_j = r_{j+1};\\
(vt-v^{-1}t) \{e_{r_1 \dots r_d}\} + t^{2} \{e_{r_1\dots r_{j-1} r_{j+1} r_j\dots r_d}\},  & r_j> r_{j+1}.
\end{cases}
\label{1}
\end{eqnarray}
\end{lem}

\begin{proof}
Formula (\ref{1}) similar with the one in  ~\cite[1. 12]{GL92},  whose  proof is also almost the same as one parameter of type-$\mbf A$ case.

\end{proof}

\subsection{Generators of $\mcal S$}
\label{sec-partial order}

Define a partial order $``\leq"$ on $\Theta_d$  by $\Aa\leq \B$ if $\mcal O_{\Aa} \subset \overline{\mcal O}_{\B}$.
For any $\Aa=(a_{ij})$ and $\B=(b_{ij})$ in $\Xi_{\mbf d}$,  we say that $\Aa \preceq \B$ if and only if
the following two conditions hold.
\begin{align}\label{partial-order}
 \sum_{r\leq i,  s\geq j} a_{rs}  & \leq \sum_{r\leq i,  s\geq j} b_{rs},  \quad \forall i<j. \\
  \sum_{r\geq i,  s\leq j} a_{rs}  & \leq \sum_{r\geq i,  s\leq j} b_{rs},  \quad \forall i>j.
\end{align}
The  relation $``\preceq" $ defines a second partial order on  $\Theta_d$.
 We say that $\Aa \prec \B$ if $\Aa \preceq \B$ and at least one of the inequalities in (\ref{partial-order}) is strict.
We shall denote by ``$\{ \m\}$+ lower terms" an element in $\mcal S$ which is equal to $\{\m\}$
plus a linear combination of $\{\m'\}$ with $\m' \prec \m$.
By Proposition (\ref{cor5}),  we have

\begin{cor} \label{cor7}
Assume that $1 \leq h <n$, $1 \leq h \leq n$,  $M = (m_{i, j}) \in \theta_{ d}$.

$(a)$ Assume that $\ m_{h, j}=0,  \forall j > k ,  m_{h+1, j}=0,  \forall j \geq k$.  Let $r = m_{h, k}$,  $\Aa=(a_{ij}) \in \Xi_{\mbf d}$ satisfies  the following two conditions:
$\ a_{h, k}=0,  a_{h+1, k}=r,    , a_{i, j} = m_{i, j}  {\rm for} \ {\rm all}\  {\rm other}\ $ $i, j. $
If   $\B $  is subject to  $\B-rE_{h, h+1}$ is diagonal,
 $\co(\B)=\ro(\Aa)$,  then
$$\{\B\} * \{\Aa\}=\{ M \} +{\rm lower\ terms}. $$

$(b)$ Assume that $\ m_{h, j}=0,  \forall j \leq k ,  m_{h+1, j}=0,  \forall j < k$.  Let $r = m_{h+1, k}$,  $\Aa=(a_{ij}) \in \theta_{ d}$ satisfies  the following two conditions:
$\ a_{h, k}= r,  a_{h+1, k}=0,    , a_{i, j} = m_{i, j}  {\rm for} \ {\rm all}\  {\rm other}\ $ $i, j. $
If   $\C $  is subject to  $\C-rE_{h, h+1}$ is diagonal,
 $\co(\C)=\ro(\Aa)$,  then
$$\{\C\} * \{\Aa\}=\{ M \} +{\rm lower\ terms}. $$
\end{cor}

\begin{proof}
In case (a),  from the proof of the  ~\cite[3. 8]{BLM90},  we have that $\{M\}$ is correspondence to $\mbf t = (0,  \cdots0, ,  R,  0, \cdots,  0)$ , where $R$ is in the k place.  Therefore,  $\alpha(t)
 = \sum_{j\geq k} a_{h, j} t_k + \sum_{j>k} a_{h+1,  k} t_k - \sum_{j<l} t_j t_l = 0. $ Then (a) follows.

In case (b),  we have we have that $\{M\}$ is correspondence to $\mbf t = (0,  \cdots0, ,  R,  0, \cdots,  0)$,  where $R$ is in the k place.  Therefore,
$\alpha'(t)
=\sum_{j\leq l} a_{h+1, j} t_l + \sum_{j<l} a_{hj} t_l - \sum_{j<l} t_j t_l = 0. $Then (b) follows.
\end{proof}

\begin{thm}\label{thm1}
For any $\Aa=(a_{ij})\in \Theta_{\mbf d}$.  The following identity holds in $\mcal S$
$$\prod_{1\leq i \leq h < j \leq n} \{D_{i, h, j} + a_{i, j}E_{h, h+1} \}*\prod_{1\leq j \leq h < i \leq n} \{D_{i, h, j} + a_{i, j}E_{h+1, h} \}  = \{\Aa\}+{\rm lower\ terms}, $$
where the product is taken in the following order.
The factors in the first product are taken in the following order: $(i,  h,  j)$ comes before $(i',  h',  j')$ if either $j > j'$ or $j = j'$,  $h - i < h' - i'$,  or $j = j',  h - i = h' - i',  i' > i$.
The factors in the second product are taken in the following order: $(i,  h,  j)$ comes before $(i',  h',  j')$ if either $i < i'$ or $i = i'$,  $h - j > h' - j'$,  or $i = i',  h - j = h' - j',  j' < j$.
The matrices $D_{i, h, j}$ are diagonal with entries in $\mathbb{N}$.  Which are uniquely determined.
\end{thm}

\begin{proof}
The proof of this theorem is similar to the \cite[3. 9]{BLM90}.
\end{proof}

We have immediately

\begin{cor} \label{S-M}
The products  $\m_{\A}=\prod_{1\leq i \leq h < j \leq n} \{D_{i, h, j} + a_{i, j}E_{h, h+1} \}*\prod_{1\leq j \leq h < i \leq n} \{D_{i, h, j} + a_{i, j}E_{h+1, h} \} $ for any  $\A\in \Theta_{\mbf d}$ in Theorem ~\ref{thm1} form  a basis for $\mcal S$.
\end{cor}

By (\ref{eq22}),  (\ref{eq24}) and Corollary ~\ref{S-M},  we have

\begin{cor} \label{S-1}
The algebra $\mcal S$ (resp.  $\mbb Q(v) \otimes_{\mcal A} \mcal S$) is generated by the elements $[\mathfrak e]$ such that
$\mathfrak e - RE_{i,  i+1}$ (resp.  either $\mathfrak e$ or  $\mathfrak e - RE_{i,  i+1}$)  is diagonal
for some $R\in \mbb N$ and $i\in [1,  n-1]$.
\end{cor}

Observe that
$
 \ E_i =\sum t\{\B\},  \ F_i = \sum \{\C\},  \ A^{\pm 1}_a =\sum  v^{\pm d_a}t^{\pm d_a} \{\D\},  B^{\pm 1}_a =\sum  v^{\mp d_a}t^{\pm d_a} \{\D\}, \quad \forall i\in [1,  n-1],  a\in [1,  n],
$
where $\B$,   $\C$ and $\D$ run over all matrices in $\Theta_{\bf d}$ such that $\B-E_{i,  i+1}$,   $\C-E_{i+1,  i}$ and $\D$ are diagonal,  respectively,  and  $d_a$ is  the $(a,  a)$-entry of the matrix in $\D$.
We have  the following corollary by Corollary ~\ref{S-1}.

\begin{cor} \label{S-generator}
The algebra $\mbb Q(v, t) \otimes_{\mcal A} \mcal S$ is generated by the functions $E_i$,  $F_i$,  $A_a^{\pm 1}$,  $B_{a}^{\pm 1}$
for any $i\in [1,  n-1]$,  $a\in [1,  n]$ .
\end{cor}

\section{The limit algebra $\mcal K$ }
\label{algebra-KD}

\subsection{Stabilization} \label{Stab}

Let $I$ be the identity matrix.  We set
$
{}_pA=A + pI.
$
Let $\widetilde{\Theta}$  be the set of all $n\times n$ matices with integer entries such that the entries off diagonal are $\geq 0$.

Let
$$\mcal K= \mbox{span}_{\mcal A} \{ \{\A\} | \A\in \wt{\Theta} \},
$$
where the notation $\{\A\}$ is a formal  symbol.
Let $v', t'$ be a independent indeterminates,  and we denote by $\mfk R$ the ring $\mathbb{Q}(v, t)[v', t']$ .

\begin{prop}
\label{prop5}
Suppose that $\A_1,  \A_2, \cdots,  \A_r \ (r\geq 2)$ are matrices in $\wt{\Theta}$
such that ${\rm co}(\A_i)={\rm ro}(\A_{i+1})$  for $1\leq i \leq r-1$.
There exist $\z_1,  \cdots,  \z_m\in \wt{\Theta}$,  $G_j(v, v', t, t')\in \mfk R$ and $p_0\in \mbb N$ such that in $\mcal S_d$ for some $d$,   we have
$$[{}_p \A_1] * [{}_p\A_2] * \cdots *[{}_p \A_r]=\sum_{j=1}^mG_j(v, v^{-p}, t, t^{p})[{}_p \z_j], \quad
\forall  p\geq p_0. $$
\end{prop}

\begin{proof}
The proof is essentially the same as the one for Proposition 4. 2 in \cite{BLM90}
by using Corollary \ref{cor5} and Theorem \ref{thm1}.
The main difference is that  we should give how  the twists  $\alpha (t)$ and  $ \alpha'(t)$  change when $\A$ is replaced by
${}_p \A$.

 If $r=2$ and $\A_1$ is chosen such that
 $\A_1-RE_{h, h+1}$ is a diagonal with $R\in \mbb N$,
the structure constant $G_{t}(v, v', t, t')$ is defined by
$$G_{t}(v, v', t, t')=v^{\beta(t)}\prod_{\overset{1\leq u\leq n}{u\neq n}}\overline{\begin{pmatrix}a_{h, u}+t_u\\
  t_u\end{pmatrix}}_{\!\!\!{}v}\prod_{1\leq i\leq t_h}\frac{v^{-2(a_{h, h}+t_h-i)}v'^2-1}{v^{-2i}-1}t^{\alpha(\mbf t)}t'^{2\sum_{h \geq u}t_u}. $$
Similaryly,
if  $r=2$ and $\A_1$ is chosen such that
$\A_1-RE_{h+1, h}$ is  diagonal with $R\in \mbb N$,
the structure constant $G_{t}(v, v', t, t')$ is defined by
\[
G_{t}(v, v', t, t')=v^{\beta'(t)}
\prod_{1\leq u\leq n, u\neq h+1}\overline{\begin{pmatrix}a_{h+1, u}+t_u\\
  t_u\end{pmatrix}}_{\!\!\!{}v}\\
  \prod_{1\leq t\leq t_{h+1}}\frac{v^{-2(a_{h+1, h+1}+t_{h+1}-i+1)}v'^2}{v^{-2i}-1}t^{\alpha (t)}t'^{\sum_{h < u t_u}},
\]

Keep in mind  the above modifications,  the rest of the proof for Proposition 4. 2 in \cite{BLM90} can be repeated here.
\end{proof}

By specialization $v', t'$ at $v'=1, t' = 1$,  there is a unique associative $\mcal A$-algebra structure on $\mcal K$,  without unit,   where
 the product is given by
 $$\{\A_1\} \cdot \{\A_2\}\cdot \dots \cdot  \{\A_r\} =\sum_{j=1}^m G_j(v, 1, t, 1)[\z_j]$$
 if $\A_1, \cdots,  \A_r$ are as in Proposition \ref{prop5}.

Let $\A$ and  $\B \in \wt{\Theta}$ be chosen such that $\B -rE_{h, h+1}$
is  diagonal for some $1\leq h < n,  r\in \mbb N$
satisfying ${\rm co}(\B)={\rm ro}(\A)$ .  Then we have
\begin{equation}\label{eq58}
\{ \B\} \cdot  \{\A\}
 =\sum_{t} v^{\beta(t)}t^{\alpha(t)}\prod_{u=1}^N
\overline{\begin{pmatrix}a_{hu}+t_u\\ t_u \end{pmatrix}}_{\!\!\!{}v}\
\{\A_{t}\},
\end{equation}
where the sum is taken over all $t=(t_u)\in \mbb N^N$ such that
$\sum_{u=1}^n t_u=r$ and $t_u \leq a_{h+1, u}$ for all $u \neq h +1$,
$\alpha(t), \beta(t)$, $\A_{t} \in \wt{\Theta}$ are defined in (\ref{eq22}).

Similarly,  if $\A,  \C \in \widetilde{\Theta}$  are chosen such that $\C-rE_{h+1, h}$
is diagonal  for some $1\leq h< n,  r\in \mbb N$ satisfying ${\rm co}(\C)={\rm ro}(\A)$ ,  then we have
\begin{align}\label{eq57}
\{\C\} \cdot \{\A\}
 =\sum_{t}v^{\beta'(t)}t^{\alpha'(t)}\prod_{u=1}^N
\overline{\begin{pmatrix} a_{h+1, u}+t_u\\ t_u\end{pmatrix}}_{\!\!\!{}v}
\ \{\A(h,  t) \},
\end{align}
where  the sum is  taken over all $t=(t_u)\in \mbb N^N$ such that $\sum_{u=1}^n t_u=r$ and$t_u \leq \A_{h, u}$ for all$u \neq h$.
$, \alpha'(t), \beta'(t)$,  $\A(h,  t) \in \widetilde{\Theta}$ are defined in (\ref{eq24}).

\subsection{The algebra $\mcal U$}\label{sec6}

In this section,  we shall define a new algebra $\mcal U$ in the  completion of $\mcal K$ similar to  ~\cite[Section 5]{BLM90}.

Let $\hat{\mcal K}$ be the $\mbb Q(v, t)$-vector space of all formal sum
$\sum_{\A\in \tilde{\Theta}}\xi_{\A} \{\A\}$ with $\xi_{\A}\in \mbb Q(v, t)$ and  a locally finite property,  i. e. ,
for any ${\mbf t}\in \mbb Z^n$,  the sets $\{\A\in \tilde{\Theta}|{\rm ro}(\A)={\mbf t},  \xi_{\A} \neq 0\}$
and
$\{\A\in \widetilde{\Theta} | {\rm co}(\A)={\mbf t},  \xi_{\A} \neq 0\}$ are finite.
The space $\hat{\mcal K}$ becomes an  associative algebra over $\mbb Q(v, t)$
 when equipped  with  the following multiplication:
$$
\sum_{\A\in \tilde{\Xi}_{\mbf D}} \xi_{\A} \{\A\}   \cdot \sum_{\B \in \tilde{\Xi}_{\mbf D}} \xi_{\B} \{\B\}
=\sum_{\A,  \B} \xi_{\A} \xi_{\B} \{\A\} \cdot \{\B\},
$$
where the product $\{\A\} \cdot \{\B\}$ is taken  in $\mcal K$.

Observe that the algebra $\hat{\mcal K}$ has a unit element $\sum\{\md\}$,  the  summation  of  all diagonal matrices.

We define the following  elements in $\hat{\mcal K}$.
For any nonzero  matrix $\A \in \wt{\Theta}$,
let $\hat{\A}$ be the matrix obtained
by replacing diagonal entries of $\A$ by zeroes.
We set
$$
\Theta^{0}= \{ \hat{\A} | \A\in \wt{\Theta} \}.
$$

For any $\hat{\A}$ in $\Theta^{0}$ and ${\mbf j}=(j_1, \cdots,  j_n)\in \mbb Z^n$,  we define
\begin{equation} \label{1wtA}
\hat{\A} ({\mbf j})=\sum_{\lambda}v^{\lambda_1j_1+\cdots+\lambda_{n}j_{n}}t^{\lambda_1|j_1|+\cdots+\lambda_{n}|j_{n}|}\{ \hat{\A} + D_{\lambda} \}\quad
\end{equation}
where the  sum runs through all $\lambda=(\lambda_i)\in \mbb Z^n$ such that
$\hat{\A} + D_{\lambda} \in \wt{\Theta}$, where $D_{\lambda}$ is the diagonal matrices with diagonal entries$(\lambda_i). $

For $i\in [1, n-1]$,  let
\begin{equation*}
E_i=E_{i, i+1}(0)\quad{\rm and}\quad
F_i=E_{i+1,  i}(0).
\end{equation*}

Let $\mcal U$ be the subalgebra of $\hat{\mcal K}$ generated by $E_i,  F_i,  0(\mbf j)$ for all $i\in [1, n-1]$ and $\mbf j\in \mbb Z^n$.

\begin{prop}\label{prop-a}
The following relations hold in $\mcal U$.
 \allowdisplaybreaks
\begin{eqnarray}
\label{1-i}
&&0(\mbf j)0(\mbf j')=0(\mbf j')0(\mbf j), \\
\label{1-ii}
&\texttt{}&0(\mbf j)E_h=v^{j_h-j_{h+1}}t^{|j_h|-|j_{h+1}|}E_h0(\mbf j), \
       0(\mbf j)F_h=v^{-j_h+j_{h+1}}t^{-|j_h|+|j_{h+1}|}F_h0(\mbf j), \\
 \label{1-iii}
  &&t(E_hF_h-F_hE_h)=(v-v^{-1})^{-1}(0(\underline h-\underline{h+1})-0(\underline{h+1} -\underline h)), \\
\label{1-iv}
  & &E_i^2E_{i+1} - (vt + v^{-1}t)E_iE_{i+1}E_{i} + t^{2}E_{i+1}E_{i}^{2} = 0, \\
  & &t^{2}E_{i+1}^2E_{i} - (vt + v^{-1}t)E_{i+1}E_{i}E_{i+1} + E_{i}E_{i+1}^{2} = 0, \\
  & &F_i^2F_{i+1} - (vt^{-1} + v^{-1}t^{-1})F_iF_{i+1}F_{i} + t^{-2}F_{i+1}F_{i}^{2} = 0, \\
  & &t^{-2}F_{i+1}^2F_{i} - (vt^{-1} + v^{-1}t^{-1})F_{i+1}F_{i}F_{i+1} + F_{i}F_{i+1}^{2} = 0.
\end{eqnarray}
where $\mbf j,  \mbf j'\in \mbb Z^n$,  $h,  i,  j\in [1,  n]$ and  $\underline i \in \mbb N^N$ is the vector whose $i$-th entry is 1 and 0 elsewhere.
\end{prop}

\begin{proof}
We show (\ref{1-ii}).
\begin{equation*}
\label{eq64}
\begin{split}
0(\mbf j)E_h
&=\textstyle\sum_{\lambda}v^{\sum \lambda_kj_k}t^{\sum \lambda_k|j_k|}\{D_{\lambda}\}
\sum_{\lambda'}\{E_{h, h+1} + D_{\lambda'}\}\\
&\textstyle =\sum_{\lambda'}v^{\sum \lambda_k'j_k + j_h}t^{\sum \lambda_k'|j_k| + |j_h|}\{E_{h, h+1} + D_{\lambda'}\},
\end{split}
\end{equation*}
where the sums run through in an obvious range by the definition in (\ref{1wtA}).
\begin{equation*}
\begin{split}
E_h 0(\mbf j)
&=\textstyle \sum_{\lambda,  \lambda' }v^{\sum \lambda_kj_k}t^{\sum \lambda_k|j_k|}\{E_{h, h+1} + D_{\lambda'}\}\{D_{\lambda}\}\\
 &=\sum_{\lambda' }v^{\sum \lambda_k'j_k + j_{h+1}}t^{\sum \lambda_k'|j_k| + |j_{h+1}|}\{E_{h, h+1} + D_{\lambda'}\}.
 \end{split}
 \end{equation*}
 So we have the first identity in (\ref{1-ii}).  All other identities in (\ref{1-i}) and (\ref{1-ii}) can be shown similarly.

We show (\ref{1-iii}).
 we have
\begin{equation*}
\begin{split}
E_hF_h
&= \textstyle\sum_{\lambda, \lambda'}\{E_{h, h+1} + D_{\lambda}\}
\{E_{h+1, h} + D_{\lambda'}\}\\
&\textstyle=\sum_{\lambda}\{E_{h, h+1} + D_{\lambda}\}
\{E_{h+1, h} + D_{\lambda}\}\\
&\textstyle=\sum_{\lambda}(v^{\lambda_h - \lambda_{h+1}}t^{\lambda_h + \lambda_{h + 1}}\overline{\begin{pmatrix}\lambda_h+1\\ 1 \end{pmatrix}}_{\!\!\!{}v}\{D_{\lambda} + E_{h, h}\}\\
&+\{E_{h+1, h} + E_{h, h+1} + D_{\lambda} - E_{h+1, h+1}\}).
\end{split}
\end{equation*}
Similarly,
\begin{equation*}
\begin{split}
F_hE_h
&\textstyle=\sum_{\lambda}\{E_{h+1, h} + D_{\lambda}\}
\{E_{h, h+1} + D_{\lambda}\}\\
&\textstyle=\sum_{\lambda}(v^{\lambda_{h+1} - \lambda_{h}}t^{\lambda_h + \lambda_{h + 1}}\overline{\begin{pmatrix}\lambda_{h+1}+1\\ 1 \end{pmatrix}}_{\!\!\!{}v}\{D_{\lambda} + E_{h+1, h+1}\}\\
&+\{E_{h+1, h} + E_{h, h+1} + D_{\lambda} - E_{h, h}\}).
\end{split}
\end{equation*}
Therefore,
\begin{equation*}
\begin{split}
t(E_hF_h - F_hE_h)
&\textstyle=\sum_{\lambda} \frac{v^{\lambda_h - \lambda_{h + 1} }t^{\lambda_h + \lambda_{h + 1}}-v^{\lambda_{h+1}- \lambda_{h}}t^{\lambda_h +\lambda_{h + 1}}}{v-v^{-1}}\{D_{\lambda}\}\\
&\textstyle=(v-v^{-1})^{-1}(0(\underline h-\underline{h+1})-0(\underline{h+1} -\underline h)).
\end{split}
\end{equation*}
At last,  We show (\ref{1-iv}).
\begin{equation*}
\begin{split}
E_h^2E_{h+1}
&\textstyle=\sum_{\lambda} vt(v^{-2} + 1)\{D_{\lambda} + E_{h, h+1} + E_{h, h+2}\} \\&+ \sum_{\lambda} v^{-1}t^3(v^{-2} + 1)\{D_{\lambda} + E_{h+1, h+2} + 2E_{h, h+1}\};
\end{split}
\end{equation*}
\begin{equation*}
\begin{split}
E_hE_{h+1}E_h
&\textstyle=\sum_{\lambda} t^2(v^{-2} + 1)\{D_{\lambda} + 2E_{h, h+1} + E_{h, h+2}\} \\&+ \sum_{\lambda}\{D_{\lambda} + E_{h, h+1} + E_{h, h+2}\};
\end{split}
\end{equation*}

\begin{equation*}
\begin{split}
E_{h+1}E_h^2
&\textstyle=\sum_{\lambda} vt(v^{-2} + 1)\{D_{\lambda} + 2E_{h, h+1} + E_{h+1, h+2}\}.
\end{split}
\end{equation*}
Then the first identity of \ref{1-iv} follows.  all other identities can be shown similarly.

\end{proof}

The Corollary  directly follows.
\begin{cor}
The assignment $E_i\mapsto tE_i$,  $F_i\mapsto F_i$,  $A_a \mapsto 0(\underline a)$
 and $B_a \mapsto 0(-\underline a)$,   for any $i\in [1, n-1]$,  $a\in [1,  n]$ ,
 defines a  algebra isomorphism $\Upsilon: U_{v, t}(gl_n)\rightarrow \mcal U$ .
\end{cor}

\section{Schur dualities for two parameter case of type $\mbf A_d$}

In this section,  we shall formulate algebraically
the dualities between algebras $U_{v, t}(gl_n)$ and the two parameter Iwahori-Hecke algebras $H_d(v, t)$of type $\mbf A_{d}$.

Let $\mbf V$ be a vector space over $\mbb Q(v,  t)$ of dimension $n$.
We fix a basis $( \mbf v_i)_{ 1\leq i\leq n}$ for $\mbf V$.
Let $\mbf V^{\otimes d}$ be the $d$-th tensor space of $\mbf V$.
Thus we have a basis
$(\mbf v_{r_1}\otimes \cdots\otimes \mbf v_{r_d})$,   where $r_1, \cdots,  r_d\in [1,  n]$,   for the tensor space $\mbf V^{\otimes d}$.

For a sequence $\mbf r=(r_1, \cdots,  r_d)$,  we write $\mbf v_{\mbf r}$ for $\mbf v_{r_1}\otimes \cdots\otimes \mbf v_{r_d}$.

For a sequence $\mbf r$ and a fixed integer $p\in [1,  d]$,  we define  the sequence $\mbf r'_p$ and $\mbf r''_p$ by
\[
(\mbf r'_p)_j=
\begin{cases}
r_j,  & j\neq p, \\
r_p -1,  & j=p
\end{cases}
\quad
\mbox{and}
\quad
(\mbf r''_p)_j=
\begin{cases}
r_j,  & j\neq p, \\
r_p +1,  & j=p,
\end{cases}
\]

\begin{lem}
\label{U-action-V}
There has  a left  $\U_{v,  t}(gl_n)$-action on $\mbf V^{\otimes d}$ defined by,  for any $i\in [1,  n - 1]$,  $a\in [1,  n]$,
 \begin{align}
 E_i \cdot \mbf v_{\mbf r}  &=\sum_{1\leq p\leq d:r_p=i+1}
v^{\sum_{j>p}\delta_{i,  r_j} - \delta_{i + 1,  r_j}}t^{1+\sum_{j>p}\delta_{i,  r_j} + \delta_{i + 1,  r_j}} \mbf v_{\mbf r'_p},   \nonumber \\
F_i \cdot \mbf v_{\mbf r}  &=\sum_{1\leq p\leq d:r_p=i}
v^{\sum_{j<p}\delta_{i+1,  r_j} - \delta_{i,  r_j}}t^{\sum_{j<p}\delta_{i,  r_j} + \delta_{i + 1,  r_j}} \mbf v_{\mbf r''_p},   \nonumber \\
A_a^{\pm 1} \cdot \mbf{v_r}  &=v^{\pm \sum_{1\leq j\leq d} \delta_{a,  r_j}} t^{\pm \sum_{1\leq j\leq d} \delta_{a,  r_j}} \mbf{v_r},  \nonumber\\
B_a^{\pm 1} \cdot \mbf{v_r}  &=v^{\mp \sum_{1\leq j\leq d} \delta_{a,  r_j}} t^{\pm \sum_{1\leq j\leq d} \delta_{a,  r_j}} \mbf{v_r}.  \nonumber
\end{align}
\end{lem}

The lemma  follows
Proposition ~\ref{prop3},   and Corollary ~\ref{cor9}.

\begin{lem}

There has  a right $\Hd$-action on $\mbf V^{\otimes d}$ given by,  for $1\leq j\leq d-1$,
\begin{eqnarray}
&\mbf v_{r_1 \dots r_{d}} T_j =
\begin{cases}
\mbf v_{r_1\dots r_{j-1} r_{j+1} r_j\dots r_{d}} ,   &  r_j < r_{j+1};\\
vt \mbf v_{r_1\dots r_{d}},  &  r_j = r_{j+1};\\
(vt-v^{-1}t) \mbf v_{r_1\dots r_{d}} + t^2 \mbf v_{r_1 \dots r_{j-1} r_{j+1} r_j  \dots   r_{d}},  & r_j> r_{j+1}.
\end{cases}
\end{eqnarray}
\end{lem}

This lemma follows  Lemmas ~\ref{Geometric-Hecke} and  ~\ref{H-action}.

We now can state the  duality.

\begin{prop}
The left $U_{v, t}(gl_n)$-action in Lemma ~\ref{U-action-V} and the right $\Hd$-action in Lemma ~\ref{H-action}  on $\mbf V^{\otimes d}$ are commuting.   They form a double centralizer for $n\geq d$,  i. e. ,
$$
\Hd \simeq \End_{\U} (\mbf V^{\otimes d})
\quad
\mbox{and}
\quad
U_{v, t}(gl_n) \rightarrow \End_{\Hd} (\mbf V^{\otimes d}) \ is \  surjective .
$$
\end{prop}

The proposition follows from  the previous two lemmas,  Lemma ~\ref{eq34},   Proposition ~\ref{prop3} and Corollary ~\ref{S-generator}.

\subsection{Galois descend approach}

Let $G = Gal(\mathbb{Q}(v, t)/\mathbb{Q}(r, s))$, $r=vt, s=v^{-1}t. $ It is easy to know \ $G\cong S_2 $ \ which is generated by $\sigma$.  $G$ act on $U_{v, t}(gl_n)$ given by a $\mathbb{Q}$ algebra homomorphism $\sigma :U_{v, t}(gl_n)\rightarrow U_{v, t}(gl_n)$; $E_i \mapsto -E_i,  F_i \mapsto F_i,  K_i \mapsto K_i, K_i' \mapsto K_i', v \mapsto -v, t \mapsto -t. $ $G$ can be also act on $V^{\otimes k}$ which is given by $\sigma :V^{\otimes k}\rightarrow  V^{\otimes k}$; $v_{i_{1}}\otimes \cdots \otimes v_{i_{k}}\mapsto v_{i_{1}}\otimes \cdots \otimes v_{i_{k}},  v\mapsto -v,  t \mapsto -t. $ By the directly compute.  we have the following lemma.

\begin{lem}
The $G$-actions on$(U_{v, t}(gl_n),  V^{\otimes k})$ is compatible. That is $\sigma(av) = \sigma(a)\sigma(v)$, $\forall a \in U_{v, t}(gl_n),  v \in V$.
\end{lem}

\begin{proof}
We only need to check  the identities $\sigma(av) = \sigma(a)\sigma(v)$ on the generators.  By the lemma \ref{U-action-V}.  The result is obvious.
\end{proof}

Though the above lemma we know there is a $G$-action on $H_k(v, t)$ which is given by $\sigma : H_k(v, t) \mapsto H_k(v, t)$;$T_i \mapsto T_i,  v \mapsto -v,  t \mapsto - t$.

\begin{thm}
$(U_{v, t}(gl_n)^G,  {V^{\otimes k}}^G, H_k(v, t)^G)$ is a shur-weyl tripple. and $U_{v, t}(gl_n)^G \cong  U_{r, s}(gl_n)$,  ${V^{\otimes k}}^G$ is a $n^k$ dimension vector space over $\mathbb{Q}(r, s)$,  $H_k(v, t)^G \cong H_k(r, s)$.
\end{thm}

\begin{proof}
\end{proof}

\begin{rem}
$H_k(r, s)$ is a unital associate algebra over $\mathbb{Q}(r, s)$ with generators $\widetilde{T_i}$, $1 \leq i < k$subject to the following ralations:
 \begin{eqnarray*}
  (1)& & \widetilde{T_i}\widetilde{T_{i+1}}\widetilde{T_{i}}=\widetilde{T_{i+1}}\widetilde{T_{i}}\widetilde{T_{i+1}},  1\leq i <k. \\
  (2)& &\widetilde{T_i}\widetilde{T_j} = \widetilde{T_j}\widetilde{T_i},  {\rm if} |i - j| \geq 2. \\
  (3)& &(\widetilde{T_i} - r)(\widetilde{T_i} + s) = 0,  \forall i.
\end{eqnarray*}
$U_{r, s}(gl_n)$ is a $\mathbb{Q}(r, s)$ algebra generated by $\widetilde{E_i}, \widetilde{F_i}, \widetilde{K_i}, \widetilde{K_i'}$.
\end{rem}

\section{ Two new quantum group  $\widetilde{U_{v, t}(gl_n)^m}$ and $\widehat{U_{v, t}(gl_n)^m}$ }
In order to give the comultiplication in the two parameter case of two new quantum group appeared in \cite{FL14},  we give two new quantum group  $\widetilde{U_{v, t}(gl_n)^m}$ and $\widehat{U_{v, t}(gl_n)^m}$ in this section.

 For any $i\in [1,  n-1]$,  $a\in [1,  n]$, $m \in [1, n-1]$ , we define the function $E_i,  F_i,  A_a^{\pm 1},  B_a^{\pm 1}$ to be the same function in $\mcal S$ . we further define
\begin{align}
\begin{split}
J_+ (V,  V') & =
\begin{cases}
1 ,  &\mbox{if}\; V=V' \mbox{and} |V_m| = d \ mod \ 2;\\
0,  & \mbox{otherwise}.
\end{cases}
\\
J_- (V,  V') & =
\begin{cases}
1 ,  &\mbox{if}\; V=V' \mbox{and} |V_m| = d-1 \ mod \ 2;\\
0,  & \mbox{otherwise}.
\end{cases}
\end{split}
\end{align}

All these functions are elements in $\mcal S$.

\begin{prop}
The functions $E_i,  F_i$,   $A_a^{\pm 1}$,  $B_a^{\pm 1}$ and $J_{\pm}$  in $\mcal S$,  for any $i\in [1, n-1]$,  $a\in [1, n]$ ,    satisfy the  relations in \ref{prop3} together with the following relations.
\begin{eqnarray*}
  (R1)&&J_+ + J_- = 1, \ \ J_{\alpha}J_{\beta} = \delta_{\alpha, \beta}J_{\alpha}, \ \ J_{\pm}A_a = A_aJ_{\pm}, \ \ J_{\pm}A_a = A_aJ_{\pm}, \\
      && J_{\pm}E_i = E_iJ_{\pm}, \ \ J_{\pm}F_i = F_iJ_{\pm},  i \neq m;\ \ J_{\pm}E_m =E_mJ_{\mp},  J_{\pm}F_m =F_mJ_{\mp};\\
\end{eqnarray*}
\end{prop}

\begin{cor}
The algebra $\mbb Q(v, t) \otimes_{\mcal A} \mcal S$ is generated by the functions  $E_i,  F_i$,   $A_a^{\pm 1}$ ,  $B_a^{\pm 1}$, and$J_{\pm}$  in $\mcal S$,  for any $i\in [1, n - 1]$,  $a\in [1,  n]$.
\end{cor}

\subsection{Another limit algebra $\mcal K'$ }

 We set
$
{}_pA=A+2pI
$.
Let
$$\mcal K'= \mbox{span}_{\mcal A} \{ \{\A\} | \A\in \wt{\Theta} \},
$$
where the notation $\{\A\}$ is a formal  symbol.
Let $v', t'$ be a independent indeterminates,  and we denote by $\mfk R$ the ring $Q(v, t)[v', t']$ .

\begin{prop}
\label{prop5'}
Suppose that $\A_1,  \A_2, \cdots,  \A_r \ (r\geq 2)$ are matrices in $\widetilde{\Theta}$
such that ${\rm co}(\A_i)={\rm ro}(\A_{i+1})$  for $1\leq i \leq r-1$.
There exist $\z_1,  \cdots,  \z_m\in \widetilde{\Theta}$,  $G'_j(v, v', t, t')\in \mfk R$ and $p_0\in \mbb N$ such that in $\mcal S_d$ for some $d$,   we have
$$\{{}_p \A_1\} * \{{}_p\A_2\} * \cdots *\{{}_p \A_r\}=\sum_{j=1}^mG'_j(v, v^{-p}, t, t^{p})\{{}_p \z_j\}, \quad
\forall  p\geq p_0. $$
\end{prop}

By specialization $v', t'$ at $v'=1, t' = 1$,
 there is a unique associative $\mcal A$-algebra structure on $\mcal K$,  without unit,   where
 the product is given by
 $$\{\A_1\} \cdot \{\A_2\}\cdot \dots \cdot  \{\A_r\} =\sum_{j=1}^m G'_j(v, 1, t, 1)[\z_j]$$
 if $\A_1, \cdots,  \A_r$ are as in Proposition \ref{prop5'}.

Let $\A$ and  $\B \in \wt{\Theta}$ be chosen such that $\B -rE_{m, m+1}$
is  diagonal for some $ r\in \mbb N$
satisfying ${\rm co}(\B)={\rm ro}(\A)$ .  Then we have
\begin{equation}\label{eq58}
\{ \B\} \cdot  \{\A\}
 =\sum_{t} v^{\beta(t)}t^{\alpha(t)}\prod_{u=1}^N
\overline{\begin{pmatrix}a_{hu}+t_u\\ t_u \end{pmatrix}}_{\!\!\!{}v}\
\{\A_{t}\},
\end{equation}
where the sum is taken over all $t=(t_u)\in \mbb N^n$ such that
$\sum_{u=1}^n t_u=r$ and $t_u \leq a_{m+1, u}$,
$\alpha(t), \beta(t)$,  $\A_{t} \in \widetilde{\Theta}$ are defined in (\ref{eq22}),
 .

Similarly,  if $\A,  \C \in \widetilde{\Theta}$  are chosen such that $\C-rE_{m+2, m+1}$
is diagonal  for some $1\leq h< n,  r\in \mbb N$ satisfying ${\rm co}(\C)={\rm ro}(\A)$ ,  then we have
\begin{align}\label{eq57}
\{\C\} \cdot \{\A\}
 =\sum_{t}v^{\beta'(t)}t^{\alpha'(t)}\prod_{u=1}^N
\overline{\begin{pmatrix} a_{h+1, u}+t_u\\ t_u\end{pmatrix}}_{\!\!\!{}v}
\ \{\A(h,  t) \},
\end{align}
where  the sum is  taken over all $t=(t_u)\in \mbb N^n$ such that $\sum_{u=1}^n t_u=r$and $t_u \leq a_{m+1, u}$,
$, \alpha'(t), \beta'(t)$ $\A(h,  t) \in \widetilde{\Theta}$ are defined in (\ref{eq24}).

\subsection{The algebra $\mcal U'$}

In this section,  we shall define a new algebra $\mcal U$ in the  completion of $\mcal K$ similar to  ~\cite[Section 5]{BLM90}.

Let $\hat{\mcal K}$ be the $\mbb Q(v, t)$-vector space of all formal sum
$\sum_{\A\in \tilde{\Theta}}\xi_{\A} \{\A\}$ with $\xi_{\A}\in \mbb Q(v, t)$ and  a locally finite property,  i. e. ,
for any ${\mbf t}\in \mbb Z^n$,  the sets $\{\A\in \tilde{\Theta}|{\rm ro}(\A)={\mbf t},  \xi_{\A} \neq 0\}$
and
$\{\A\in \widetilde{\Theta} | {\rm co}(\A)={\mbf t},  \xi_{\A} \neq 0\}$ are finite.
The space $\hat{\mcal K}$ becomes an  associative algebra over $\mbb Q(v, t)$
 when equipped  with  the following multiplication:
$$
\sum_{\A\in \widetilde{\Theta}} \xi_{\A} \{\A\}   \cdot \sum_{\B \in \widetilde{\Theta}} \xi_{\B} \{\B\}
=\sum_{\A,  \B} \xi_{\A} \xi_{\B} \{\A\} \cdot \{\B\},
$$
where the product $\{\A\} \cdot \{\B\}$ is taken  in $\mcal K$.
This is shown in exactly the same as ~\cite[Section 5]{BLM90}.

Observe that the algebra $\hat{\mcal K}$ has a unit element $\sum\{\md\}$,  the  summation  of  all diagonal matrices.

We define the following  elements in $\hat{\mcal K}$.
For any nonzero  matrix $\A \in \wt{\Theta}$,
let $\hat{\A}$ be the matrix obtained
by replacing diagonal entries of $\A$ by zeroes.
We set
$
\Theta^{0}= \{ \hat{\A} | \A\in \wt{\Theta} \}.
$

For any $\hat{\A}$ in $\Theta^{0}$ and ${\mbf j}=(j_1, \cdots,  j_n)\in \mbb Z^n$,  we define
\begin{equation} \label{wtA}
\hat{\A} ({\mbf j})=\sum_{\lambda}v^{\lambda_1j_1+\cdots+\lambda_{n}j_{n}}t^{\lambda_1|j_1|+\cdots+\lambda_{n}|j_{n}|}\{ \hat{\A} + D_{\lambda} \}\quad
\end{equation}
where the  sum runs through all $\lambda=(\lambda_i)\in \mbb Z^n$ such that
$\hat{\A} + D_{\lambda},   \in \wt{\Theta}$,  where $D_{\lambda}$ is the diagonal matrices with diagonal entries $(\lambda_i). $

And we also define
$$
J_{+} = \sum_{\lambda \in S_0}\{ D_{\lambda} \}, \
J_{-} = \sum_{\lambda \in S_1}\{ D_{\lambda} \},
$$
Where $S_0 = \{\lambda |  \sum_{i = 1}^{m} \lambda_i \equiv \sum_{i = 1}^{n} \lambda_i\ mod \ 2 \}$,
$S_1 = \{\lambda | \sum_{i = 1}^{m} \lambda_i \equiv \sum_{i = 1}^{n} \lambda_i -1 \ mod \ 2 \}$

For $i\in [1, n-1]$,  let
\begin{equation*}
E_i=E_{i, i+1}(0)\quad{\rm and}\quad
F_i=E_{i+1,  i}(0).
\end{equation*}

Let $\mcal U'$ be the subalgebra of $\hat{\mcal K}$ generated by $E_i,  F_i,  0(\mbf j)$ , $J_{\pm}$ for all $i\in [1, n-1]$ and $\mbf j\in \mbb Z^n$.

\begin{prop}
The following relations hold in $\mcal U'$.
 \allowdisplaybreaks
\begin{eqnarray}
    &&J_+ + J_- = 1, \ \ J_{\alpha}J_{\beta} = \delta_{\alpha, \beta}J_{\alpha}, \ \ J_{\pm}0(\mbf j) = 0(\mbf j)J_{\pm}, \\
      && J_{\pm}E_i = E_iJ_{\pm}, \ \ J_{\pm}F_i = F_iJ_{\pm},  i \neq m;\ \ J_{\pm}E_m =E_mJ_{\mp},  J_{\pm}F_m =F_mJ_{\mp};\\
&&0(\mbf j)0(\mbf j')=0(\mbf j')0(\mbf j), \\
&\texttt{}&0(\mbf j)E_h=v^{j_h-j_{h+1}}t^{|j_h|-|j_{h+1}|}E_h0(\mbf j), \
       0(\mbf j)F_h=v^{-j_h+j_{h+1}}t^{-|j_h|+|j_{h+1}|}F_h0(\mbf j), \\
  &&t(E_hF_h-F_hE_h)=(v-v^{-1})^{-1}(0(\underline h-\underline{h+1})-0(\underline{h+1} -\underline h)), \\
  & &E_i^2E_{i+1} - (vt + v^{-1}t)E_iE_{i+1}E_{i} + t^{2}E_{i+1}E_{i}^{2} = 0, \\
  & &t^{2}E_{i+1}^2E_{i} - (vt + v^{-1}t)E_{i+1}E_{i}E_{i+1} + E_{i}E_{i+1}^{2} = 0, \\
  & &F_i^2F_{i+1} - (vt^{-1} + v^{-1}t^{-1})F_iF_{i+1}F_{i} + t^{-2}F_{i+1}F_{i}^{2} = 0, \\
  & &t^{-2}F_{i+1}^2F_{i} - (vt^{-1} + v^{-1}t^{-1})F_{i+1}F_{i}F_{i+1} + F_{i}F_{i+1}^{2} = 0.
\end{eqnarray}
where $\mbf j,  \mbf j'\in \mbb Z^n$,  $h,  i,  j\in [1,  n]$ and  $\underline i \in \mbb N^N$ is the vector whose $i$-th entry is 1 and 0 elsewhere.
\end{prop}

\subsection{The algebra $\widetilde{U_{v, t}(gl_n)^m}$}

\begin{Def}
$\widetilde{U_{v, t}(gl_n)^m}$ is an associative $\mathbb{Q}(v, t)$-algebra with 1 generated by symbols $E_i,  F_i,  A_a,  B_a,  J_{\alpha}$ for all $i\in [1, n-1]$, $a \in [1,  n]$ and $\alpha \in \{+, -\}$ and subject to the following relations.
 \allowdisplaybreaks
\begin{eqnarray}
&&J_+ + J_- = 1, \ \ J_{\alpha}J_{\beta} = \delta_{\alpha, \beta}J_{\alpha}, \ \ J_{\pm}A_a = A_aJ_{\pm}, \ \ J_{\pm}B_a = B_aJ_{\pm}, \\
      && J_{\pm}E_i = E_iJ_{\pm}, \ \ J_{\pm}F_i = F_iJ_{\pm},  i \neq m;\ \ J_{\pm}E_m =E_mJ_{\mp},  J_{\pm}F_m =F_mJ_{\mp};\\
& & A_i^{\pm 1}A^{\pm 1}_j=A^{\pm 1}_jA_i^{\pm 1}, \ \ B^{\pm 1}_iB^{\pm 1}_j=B^{\pm 1}_jB^{\pm 1}_i, \\
     & & A_i^{\pm 1}B^{\pm 1}_j=B^{\pm 1}_jA_i^{\pm 1}, \ \ A_i^{\pm 1}A_i^{\mp 1}=1=B^{\pm 1}_iB^{\mp 1}_i. \\
& &A_iE_jA^{-1}_i=v^{\langle i, j\rangle}t^{\langle i, j\rangle}E_j, \ \ B_iE_jB^{-1}_i=v^{-\langle i, j\rangle}t^{\langle i, j\rangle}E_j, \\
     & &A_iF_jA^{-1}_i=v^{-\langle i, j\rangle}t^{- \langle j, i\rangle}F_j, \ \ B_iF_jB^{-1}_i=v^{\langle i, j\rangle}t^{-\langle j, i\rangle}F_j. \\
& &E_iF_j-F_j E_i=\delta_{ij}\frac{A_iB_{i+1}-B_iA_{i+1}}{vt-v^{-1}t}. \\
  & &E_i^2E_{i+1} - (vt + v^{-1}t)E_iE_{i+1}E_{i} + t^{2}E_{i+1}E_{i}^{2} = 0, \\
  & &t^{2}E_{i+1}^2E_{i} - (vt + v^{-1}t)E_{i+1}E_{i}E_{i+1} + E_{i}E_{i+1}^{2} = 0, \\
  & &F_i^2F_{i+1} - (vt^{-1} + v^{-1}t^{-1})F_iF_{i+1}F_{i} + t^{-2}F_{i+1}F_{i}^{2} = 0, \\
  & &t^{-2}F_{i+1}^2F_{i} - (vt^{-1} + v^{-1}t^{-1})F_{i+1}F_{i}F_{i+1} + F_{i}F_{i+1}^{2} = 0.
\end{eqnarray}
\end{Def}

\begin{prop} \label{Upsilon}
The assignment $E_i\mapsto E_i$,  $F_i\mapsto F_i$,  $A_a \mapsto 0(\underline a)$
,  $B_a \mapsto 0(-\underline a)$, and $J_{\alpha} \mapsto J_{\alpha}$  for any $i\in [1, n-1]$,  $a\in [1,  n]$ , $\alpha \in \{+,  -\}$
 defines a  algebra isomorphism $\Upsilon: \widetilde{U_{v, t}(gl_n)^m} \rightarrow \mcal U'$ .
\end{prop}

\subsection{Defining relations of $\mcal S$}

 For any $i\in [1,  n-1]$,  $a\in [1,  n]$, $m \in [1, n-1]$ , we define the function $E_i,  F_i,  A_a^{\pm 1},  B_a^{\pm 1}$ to be the same function in $\mcal S$ . we further define
\begin{align}
\begin{split}
J_+ (V,  V') &=
\begin{cases}
1,  &\mbox{if}\; V_m=V_{m+1}, |V_m| \equiv d \ mod\ 2 ; \\
0,  &\mbox{otherwise}.
\end{cases}
\\
J_- (V,  V') &=
\begin{cases}
1,  &\mbox{if}\; V_m=V_{m+1}, |V_m| \equiv d - 1\ mod\ 2 ; \\
0,  &\mbox{otherwise}.
\end{cases}
\\
J_0 = 1 - J_+ - J_-.
\end{split}
\end{align}

\begin{prop}\label{prop 5}
The functions $E_i,  F_i$,   $A_a^{\pm 1}$ ,  $B_a^{\pm 1}$, and$J_{\alpha}$  in $\mcal S$,  for any $i\in [1, n-1]$,  $a\in [1, n]$, $\alpha \in \{+, -, 0\}$ ,    satisfy
the  relations in proposition \ref{prop3} and the following relations.
\begin{equation}\label{relation r1}
J_+ +  J_0 +J_- = 1, J_{\alpha}J_{\beta} = \delta_{\alpha, \beta}J_{\alpha},  J_{\alpha}A_{a} = A_{a}J_{\alpha}, J_{\alpha}B_{a} = A_{a}B_{\alpha};
\end{equation}
\begin{equation}\label{relation r2}
E_iJ_{\pm} = (1 - \delta_{i, m})J_{\pm}E_i,  J_{\pm}E_i = (1 - \delta_{i, m+1})E_iJ_{\pm};
\end{equation}
\begin{equation}\label{relation r3}
F_iJ_{\pm} = (1 - \delta_{i, m+1})J_{\pm}F_i,  J_{\pm}F_i = (1 - \delta_{i, m})F_iJ_{\pm};
\end{equation}
\begin{equation}\label{relation r4}
 J_{\pm}E_mE_{m+1} =E_mE_{m+1}J_{\mp};
\end{equation}
\begin{equation}\label{relation r5}
 J_{\pm}F_{m+1}F_{m} =F_{m+1}F_{m}J_{\mp};
\end{equation}
\begin{equation}\label{relation r6}
 J_{\pm}E_{m}F_{m} - E_{m}F_{m}J_{\mp} = \frac{A_mB_{m+1}-B_mA_{m+1}}{v-v^{-1}}(J_{\pm} - J_{\mp});
\end{equation}
\begin{equation}\label{relation r7}
 J_{\pm}F_{m+1}E_{m+1} - F_{m+1}E_{m+1}J_{\mp} = \frac{B_{m+1}A_{m+2}-A_{m+1}B_{m+2}}{v-v^{-1}}(J_{\pm} - J_{\mp}).
\end{equation}
\end{prop}

\begin{proof}
The first identity in the first three rows of the relations in the proposition are straightforward.
Let $\lambda'_i=|V'_i/V'_{i-1}|$.
We show the identity $\ref{relation r5}$, by a direct calculation.
We have
\begin{equation*}
    F_{m+1}F_{m}(V, V')=\left\{\begin{array}{l}
      v^{-\lambda'_{m+2}-\lambda_{m+1}'}t^{\lambda'_{m+2}+\lambda_{m+1}'} , \quad {\rm if}\  V_m\overset{1}{\subseteq} V_m' {\rm and}\ V_{m+1}\overset{1}{\subseteq} V_{m+1}' ,  \\
      0,  \hspace{159pt} {\rm otherwise}.
    \end{array}
  \right.
  \end{equation*}

 \begin{equation*}
    J_+F_{m+1}F_{m}(V, V')=\left\{\begin{array}{l}
      v^{-\lambda'_{m+2}-\lambda_{m+1}'}t^{\lambda'_{m+2}+\lambda_{m+1}'} , \quad {\rm if}\  V_m\overset{1}{\subseteq} V_m' , \ V_{m+1}\overset{1}{\subseteq} V_{m+1}',  V_m = V_{m+1}\\
      \quad \quad \quad \quad \quad \quad \quad\quad\quad\quad\quad{\rm and} |V_m| \equiv d\ mod\ 2  ,  \\
      0,  \hspace{159pt} {\rm otherwise}.
    \end{array}
  \right.
  \end{equation*}
  \begin{equation*}
   F_{m+1}F_{m}J_-(V, V')=\left\{\begin{array}{l}
      v^{-\lambda'_{m+2}-\lambda_{m+1}'}t^{\lambda'_{m+2}+\lambda_{m+1}'} , \quad {\rm if}\  V_m\overset{1}{\subseteq} V_m' , \ V_{m+1}\overset{1}{\subseteq} V_{m+1}',  V_m = V_{m+1}\\
      \quad \quad \quad \quad \quad \quad \quad\quad\quad\quad\quad{\rm and} |V_m| \equiv d\ mod\ 2  ,  \\
      0,  \hspace{159pt} {\rm otherwise}.
    \end{array}
  \right.
  \end{equation*}
  The first part of the identity $\ref{relation r5}$ follows,  all other identities in  $\ref{relation r5}$ and $\ref{relation r4}$ can be shown similarly.

Then,
We  show the  identity $\ref{relation r7}$.
By a direct calculation,  we have
 \begin{equation*}
    F_{m + 1}E_{m + 1}(V, V')=\left\{\begin{array}{l}
     \frac{v^{2\lambda_{m+2}' - 1}}{v^2 - 1}v^{-\lambda_{m+2}' - \lambda_{m+1}' + 1} t^{\lambda_{m+2}' + \lambda_{m+1}'}, \quad {\rm if}\  V = V'\\
     v^{-\lambda_{m+2}' - \lambda_{m+1}' + 1} t^{\lambda_{m+2}' + \lambda_{m+1}'},  \quad {\rm if}\ |V_{m+1} \cap V_{m+1}'| = |V_{m+1}| - 1 = |V_{m+1}'| - 1, \\
      0,  \hspace{159pt} {\rm otherwise}.
    \end{array}
  \right.
  \end{equation*}

  \begin{equation*}
    (J_+F_{m + 1}E_{m + 1} - F_{m + 1}E_{m + 1}J_- )(V, V')=\left\{\begin{array}{l}
     \frac{v^{\lambda_{m+2}'}t^{\lambda_{m+2}'} - v^{-\lambda_{m+2}'}t^{\lambda_{m+2}'}}{v - v^{-1}}, \quad {\rm if}\  V = V', V_m = V_{m+1}, \\
    \quad \quad \quad \quad \quad \quad \quad\quad\quad\quad\quad{\rm and} |V_m| \equiv d\ mod\ 2  ,  \\
    \frac{v^{-\lambda_{m+2}'}t^{\lambda_{m+2}'}- v^{\lambda_{m+2}'}t^{\lambda_{m+2}'}}{v - v^{-1}}, \quad {\rm if}\  V = V', V_m = V_{m+1}, \\
    \quad \quad \quad \quad \quad \quad \quad\quad\quad\quad\quad{\rm and} |V_m| \equiv d - 1\ mod\ 2  ,  \\
      0,  \hspace{159pt} {\rm otherwise}.
    \end{array}
  \right.
  \end{equation*}
   \begin{equation*}
    \frac{B_{m+1}A_{m+2}-A_{m+1}B_{m+2}}{v-v^{-1}}(J_{\pm} - J_{\mp})(V, V')=\left\{\begin{array}{l}
     \frac{v^{\lambda_{m+2}'}t^{\lambda_{m+2}'} - v^{-\lambda_{m+2}'}t^{\lambda_{m+2}'}}{v - v^{-1}}, \quad {\rm if}\  V = V', V_m = V_{m+1}, \\
    \quad \quad \quad \quad \quad \quad \quad\quad\quad\quad\quad{\rm and} |V_m| \equiv d\ mod\ 2  ,  \\
    \frac{v^{-\lambda_{m+2}'}t^{\lambda_{m+2}'}- v^{\lambda_{m+2}'}t^{\lambda_{m+2}'}}{v - v^{-1}}, \quad {\rm if}\  V = V', V_m = V_{m+1}, \\
    \quad \quad \quad \quad \quad \quad \quad\quad\quad\quad\quad{\rm and} |V_m| \equiv d - 1\ mod\ 2  ,  \\
      0,  \hspace{159pt} {\rm otherwise}.
    \end{array}
  \right.
  \end{equation*}
The first part of the identity $\ref{relation r7}$ follows,  all other identities in  $\ref{relation r7}$ and $\ref{relation r6}$ can be shown similarly.

\end{proof}

\begin{cor} \label{S-new-generator}
The algebra $\mbb Q(v, t) \otimes_{\mcal A} \mcal S$ is generated by the functions  $E_i,  F_i$,   $A_a^{\pm 1}$ ,  $B_a^{\pm 1}$, and$J_{\alpha}$  in $\mcal S$,  for any $i\in [1, n - 1]$,  $a\in [1,  n]$, $\alpha \in \{+, -, 0\}$.
\end{cor}

\subsection{Limit algebra $\mcal K''$ }

Let $I' = I - E_{m+1, m+1}$ be the identity matrix.  We set
$
{}_pA=A+2pI'
$
Let
$\widetilde{\Theta}'=\{M | M\in \widetilde{\Theta},  M_{m+1, m+1} \geq 0\}$ .

Let
$$\mcal K''= \mbox{span}_{\mcal A} \{ \{\A\} | \A\in \wt{\Theta}' \},
$$
where the notation $\{\A\}$ is a formal  symbol
bearing no geometric meaning.
Let $v', t'$ be a independent indeterminates,  and we denote by $\mfk R$ the ring $Q(v, t)[v', t']$ .

\begin{prop}
\label{prop5''}
Suppose that $\A_1,  \A_2, \cdots,  \A_r \ (r\geq 2)$ are matrices in $\widetilde{\Theta}'$
such that ${\rm co}(\A_i)={\rm ro}(\A_{i+1})$  for $1\leq i \leq r-1$.
There exist $\z_1,  \cdots,  \z_m\in \widetilde{\Theta}'$,  $G'_j(v, v', t, t')\in \mfk R$ and $p_0\in \mbb N$ such that in $\mcal S_d$ for some $d$,   we have
$$\{{}_p \A_1\} * \{{}_p\A_2\} * \cdots *\{{}_p \A_r\}=\sum_{j=1}^mG'_j(v, v^{-p}, t, t^{p})\{{}_p \z_j\}, \quad
\forall  p\geq p_0. $$
\end{prop}

By specialization $v', t'$ at $v'=1, t' = 1$,
there is a unique associative $\mcal A$-algebra structure on $\mcal K$,  without unit,   where
 the product is given by
 $$\{\A_1\} \cdot \{\A_2\}\cdot \dots \cdot  \{\A_r\} =\sum_{j=1}^m G'_j(v, 1, t, 1)[\z_j]$$
 if $\A_1, \cdots,  \A_r$ are as in Proposition \ref{prop5''}.

Let $\A$ and  $\B \in \wt{\Theta}$ be chosen such that $\B -rE_{m, m+1}$
is  diagonal for some $ r\in \mbb N$
satisfying ${\rm co}(\B)={\rm ro}(\A)$ .  Then we have
\begin{equation}\label{eq58}
\{ \B\} \cdot  \{\A\}
 =\sum_{t} v^{\beta(t)}t^{\alpha(t)}\prod_{u=1}^N
\overline{\begin{pmatrix}a_{hu}+t_u\\ t_u \end{pmatrix}}_{\!\!\!{}v}\
\{\A_{t}\},
\end{equation}
where the sum is taken over all $t=(t_u)\in \mbb N^n$ such that
$\sum_{u=1}^n t_u=r$ and $t_u \leq a_{m+1, u}$,
$\alpha(t), \beta(t)$,  $\A_{t} \in \widetilde{\Theta}'$are defined in (\ref{eq22}).

Similarly,  if $\A,  \C \in \widetilde{\Theta}$  are chosen such that $\C-rE_{m+2, m+1}$
is diagonal  for some $1\leq h< n,  r\in \mbb N$ satisfying ${\rm co}(\C)={\rm ro}(\A)$ ,  then we have
\begin{align}\label{eq57}
\{\C\} \cdot \{\A\}
 =\sum_{t}v^{\beta'(t)}t^{\alpha'(t)}\prod_{u=1}^N
\overline{\begin{pmatrix} a_{h+1, u}+t_u\\ t_u\end{pmatrix}}_{\!\!\!{}v}
\ \{\A(h,  t) \},
\end{align}
where  the sum is  taken over all $t=(t_u)\in \mbb N^n$ such that $\sum_{u=1}^n t_u=r$and $t_u \leq a_{m+1, u}$,
$, \alpha'(t), \beta'(t)$ $\A(h,  t) \in \widetilde{\Theta}'$ are defined in (\ref{eq24}).

\subsection{The algebra $\mcal U''$}

In this section,  we shall define a new algebra $\mcal U''$ in the  completion of $\mcal K''$ similar to  ~\cite[Section 5]{BLM90}.

Let $\hat{\mcal K''}$ be the $\mbb Q(v, t)$-vector space of all formal sum
$\sum_{\A\in \tilde{\Theta}'}\xi_{\A} \{\A\}$ with $\xi_{\A}\in \mbb Q(v, t)$ and  a locally finite property,  i. e. ,
for any ${\mbf t}\in \mbb Z^n$,  the sets $\{\A\in \tilde{\Theta}'|{\rm ro}(\A)={\mbf t},  \xi_{\A} \neq 0\}$
and
$\{\A\in \widetilde{\Theta}' | {\rm co}(\A)={\mbf t},  \xi_{\A} \neq 0\}$ are finite.
The space $\hat{\mcal K''}$ becomes an  associative algebra over $\mbb Q(v, t)$
 when equipped  with  the following multiplication:
$$
\sum_{\A\in \tilde{\Theta}' }\xi_{\A} \{\A\}   \cdot \sum_{\B \in \tilde{\Theta}'} \xi_{\B} \{\B\}
=\sum_{\A,  \B} \xi_{\A} \xi_{\B} \{\A\} \cdot \{\B\},
$$
where the product $\{\A\} \cdot \{\B\}$ is taken  in $\mcal K''$.
This is shown in exactly the same as ~\cite[Section 5]{BLM90}.

Observe that the algebra $\hat{\mcal K''}$ has a unit element $\sum\{\md\}$,  the  summation  of  all diagonal matrices.

We define the following  elements in $\hat{\mcal K''}$.
For any nonzero  matrix $\A \in \wt{\Theta}'$,
let $\hat{\A}$ be the matrix obtained
by replacing diagonal entries of $\A$ by zeroes.
We set
$$
\Theta^{0}= \{ \hat{\A} | \A\in \wt{\Theta}^{'} \}.
$$

For any $\hat{\A}$ in $\Theta^{0}$ and ${\mbf j}=(j_1, \cdots,  j_n)\in \mbb Z^n$,  we define
\begin{equation}
\hat{\A} ({\mbf j})=\sum_{\lambda}v^{\lambda_1j_1+\cdots+\lambda_{n}j_{n}}t^{\lambda_1|j_1|+\cdots+\lambda_{n}|j_{n}|}\{ \hat{\A} + D_{\lambda} \}\quad
\end{equation}
where the  sum runs through all $\lambda=(\lambda_i)\in \mbb Z^n$ such that
$\hat{\A} + D_{\lambda},   \in \wt{\Theta}'$, where $D_{\lambda}$ is the diagonal matrices with diagonal entries$(\lambda_i). $

And we also define
\begin{equation}
J_{+} = \sum_{\lambda \in S_0}\{ D_{\lambda} \},
\end{equation}
\begin{equation}
J_{-} = \sum_{\lambda \in S_1}\{ D_{\lambda} \},
\end{equation}
\begin{equation}
J_{0} = 1 - J_{+} - J_{-}.
\end{equation}
Where $S_0 = \{\lambda | \lambda_{m+1} = 0,  \sum_{i = 1}^{m} \lambda_i \equiv \sum_{i = 1}^{n} \lambda_i\ mod \ 2 \}$,
$S_1 = \{\lambda | \lambda_{m+1} = 0,  \sum_{i = 1}^{m} \lambda_i \equiv \sum_{i = 1}^{n} \lambda_i -1 \ mod \ 2 \}$

For $i\in [1, n-1]$,  let
\begin{equation*}
E_i=E_{i, i+1}(0)\quad{\rm and}\quad
F_i=E_{i+1,  i}(0).
\end{equation*}

Let $\mcal U$ be the subalgebra of $\hat{\mcal K}$ generated by $E_i,  F_i,  0(\mbf j),  J_{\alpha}$ for all $i\in [1, n-1]$, $\mbf j\in \mbb Z^n$ and $\alpha \in \{+, -, 0\}$.

\begin{prop}\label{prop-a}
The following relations hold in $\mcal U''$.
 \allowdisplaybreaks
\begin{eqnarray}
\label{relation rr1}
&&J_+ +  J_0 +J_- = 1, J_{\alpha}J_{\beta} = \delta_{\alpha, \beta}J_{\alpha},  J_{\alpha}A_{a} = A_{a}J_{\alpha}, J_{\alpha}B_{a} = B_{a}J_{\alpha};\\
\label{relation rr2}
&&E_iJ_{\pm} = (1 - \delta_{i, m})J_{\pm}E_i,  J_{\pm}E_i = (1 - \delta_{i, m+1})E_iJ_{\pm};\\
&&F_iJ_{\pm} = (1 - \delta_{i, m+1})J_{\pm}F_i,  J_{\pm}F_i = (1 - \delta_{i, m})F_iJ_{\pm};\\
\label{relation rr4}
 &&J_{\pm}E_mE_{m+1} =E_mE_{m+1}J_{\mp};\\
&& J_{\pm}F_{m+1}F_{m} =F_{m+1}F_{m}J_{\mp};\\
\label{relation rr6}
&& J_{\pm}E_{m}F_{m} - E_{m}F_{m}J_{\mp} = \frac{A_mB_{m+1}-B_mA_{m+1}}{v-v^{-1}}(J_{\pm} - J_{\mp});\\
&& J_{\pm}F_{m+1}E_{m+1} - F_{m+1}E_{m+1}J_{\mp} = \frac{B_{m+1}A_{m+2}-A_{m+1}B_{m+2}}{v-v^{-1}}(J_{\pm} - J_{\mp}). \\
\label{eqr60}
&&0(\mbf j)0(\mbf j')=0(\mbf j')0(\mbf j), \\
\label{rA-ii}
&&0(\mbf j)E_h=v^{j_h-j_{h+1}}t^{|j_h|-|j_{h+1}|}E_h0(\mbf j), \
       0(\mbf j)F_h=v^{-j_h+j_{h+1}}t^{-|j_h|+|j_{h+1}|}F_h0(\mbf j), \\
 \label{rA-iii}
&&t(E_hF_h-F_hE_h)=(v-v^{-1})^{-1}(0(\underline h-\underline{h+1})-0(\underline{h+1} -\underline h)),  h \neq m + 1\\
\label{rA-iv}
  & &E_i^2E_{i+1} - (vt + v^{-1}t)E_iE_{i+1}E_{i} + t^{2}E_{i+1}E_{i}^{2} = 0, \\
  & &t^{2}E_{i+1}^2E_{i} - (vt + v^{-1}t)E_{i+1}E_{i}E_{i+1} + E_{i}E_{i+1}^{2} = 0, \\
  & &F_i^2F_{i+1} - (vt^{-1} + v^{-1}t^{-1})F_iF_{i+1}F_{i} + t^{-2}F_{i+1}F_{i}^{2} = 0, \\
  & &t^{-2}F_{i+1}^2F_{i} - (vt^{-1} + v^{-1}t^{-1})F_{i+1}F_{i}F_{i+1} + F_{i}F_{i+1}^{2} = 0.
\end{eqnarray}
where $\mbf j,  \mbf j'\in \mbb Z^n$,  $h,  i,  j\in [1,  n]$ and  $\underline i \in \mbb N^n$ is the vector whose $i$-th entry is 1 and 0 elsewhere.
\end{prop}

\subsection{The algebra $\widehat{U_{v, t}(gl_N)^m}$}

\begin{Def}
$\widehat{U_{v, t}(gl_n)}$ is an associative $\mathbb{Q}(v, t)$-algebra with 1 generated by symbols $E_i,  F_i,  ,  J_{\alpha}, A_a, B_a$ for all $i\in [1, n-1]$, $a \in \mbb [1, n]$ and $\alpha \in \{+, -, 0\}$ and subject to the following relations.
 \allowdisplaybreaks
\begin{eqnarray}
&&J_+ +  J_0 +J_- = 1, J_{\alpha}J_{\beta} = \delta_{\alpha, \beta}J_{\alpha},  J_{\alpha}A_{a} = A_{a}J_{\alpha}, J_{\alpha}B_{a} = B_{a}J_{\alpha};\\
&&E_iJ_{\pm} = (1 - \delta_{i, m})J_{\pm}E_i,  J_{\pm}E_i = (1 - \delta_{i, m+1})E_iJ_{\pm};\\
&&F_iJ_{\pm} = (1 - \delta_{i, m+1})J_{\pm}F_i,  J_{\pm}F_i = (1 - \delta_{i, m})F_iJ_{\pm};\\
 &&J_{\pm}E_mE_{m+1} =E_mE_{m+1}J_{\mp};\\
&& J_{\pm}F_{m+1}F_{m} =F_{m+1}F_{m}J_{\mp};\\
&& J_{\pm}E_{m}F_{m} - E_{m}F_{m}J_{\mp} = \frac{A_mB_{m+1}-B_mA_{m+1}}{v-v^{-1}}(J_{\pm} - J_{\mp});\\
&& J_{\pm}F_{m+1}E_{m+1} - F_{m+1}E_{m+1}J_{\mp} = \frac{B_{m+1}A_{m+2}-A_{m+1}B_{m+2}}{v-v^{-1}}(J_{\pm} - J_{\mp}). \\
& & A_i^{\pm 1}A^{\pm 1}_j=A^{\pm 1}_jA_i^{\pm 1}, \ \ B^{\pm 1}_iB^{\pm 1}_j=B^{\pm 1}_jB^{\pm 1}_i, \\
     & & A_i^{\pm 1}B^{\pm 1}_j=B^{\pm 1}_jA_i^{\pm 1}, \ \ A_i^{\pm 1}A_i^{\mp 1}=1=B^{\pm 1}_iB^{\mp 1}_i. \\
  & &A_iE_jA^{-1}_i=v^{\langle i, j\rangle}t^{\langle i, j\rangle}E_j, \ \ B_iE_jB^{-1}_i=v^{-\langle i, j\rangle}t^{\langle i, j\rangle}E_j, \\
     & &A_iF_jA^{-1}_i=v^{-\langle i, j\rangle}t^{- \langle j, i\rangle}F_j, \ \ B_iF_jB^{-1}_i=v^{\langle i, j\rangle}t^{-\langle j, i\rangle}F_j. \\
  & &E_iF_j-F_j E_i=\delta_{ij}\frac{A_iB_{i+1}-B_iA_{i+1}}{v-v^{-1}}. \\
  & &E_i^2E_{i+1} - (vt + v^{-1}t)E_iE_{i+1}E_{i} + t^{2}E_{i+1}E_{i}^{2} = 0, \\
  & &t^{2}E_{i+1}^2E_{i} - (vt + v^{-1}t)E_{i+1}E_{i}E_{i+1} + E_{i}E_{i+1}^{2} = 0, \\
  & &F_i^2F_{i+1} - (vt^{-1} + v^{-1}t^{-1})F_iF_{i+1}F_{i} + t^{-2}F_{i+1}F_{i}^{2} = 0, \\
  & &t^{-2}F_{i+1}^2F_{i} - (vt^{-1} + v^{-1}t^{-1})F_{i+1}F_{i}F_{i+1} + F_{i}F_{i+1}^{2} = 0.
\end{eqnarray}
\end{Def}

\newpage
\noindent{Acknowledgements: This work is supported by NSFC 11571119 and NSFC 11475178.}
\end{document}